\documentclass[draftcls]{article}

\usepackage[ruled,vlined]{algorithm2e}
\usepackage{amsmath,amssymb,amsthm}
\usepackage{mathrsfs}
\usepackage{float}
\usepackage{color}
\usepackage{hyperref}
\usepackage{amsmath}
\usepackage{float}
\usepackage{graphicx}
\usepackage{subfigure}
\usepackage{epstopdf}
\usepackage{graphicx,color,xcolor}
\usepackage{caption}
\usepackage{lineno,hyperref}
\usepackage{fontenc}
\usepackage{inputenc}
\usepackage{authblk}

\newcommand {\OM}{\Omega}
\newcommand {\GM}{\Gamma}
\title{\large \bf  CDNNs: The coupled deep neural networks for coupling of the Stokes and Darcy-Forchheimer problems
\thanks{.}}
\author[a]{Jing Yue}
\author[a]{Jian Li \thanks{Corresponding author(Jian Li) email: jianli@sust.edu.cn; jiaaanli@gmail.com.}}
\author[a]{Wen Zhang}
\affil[a]{School of Electrical and Control Engineering, School of Mathematics and Data Science, Shaanxi University of Science and Technology, Xi'an}

\date{}
\begin{document}
\maketitle
  \begin{abstract}
In this article, we present an efficient deep learning method called coupled deep neural networks (CDNNs) for coupled physical problems. Our method compiles the interface conditions of the coupled PDEs into the networks properly and can be served as an efficient alternative to the complex coupled problems. To impose energy conservation constraints, the CDNNs utilize simple fully connected layers and a custom loss function to perform the model training process as well as the physical property of the exact solution. The approach can be beneficial for the following reasons: Firstly, we sampled randomly and only input spatial coordinates without being restricted by the nature of samples. Secondly, our method is meshfree which makes it more efficient than the traditional methods. Finally, our method is parallel and can solve multiple variables independently at the same time. We give the theory to guarantee the convergence of the loss function and the convergence of the neural networks to the exact solution. Some numerical experiments are performed and discussed to demonstrate the performance of the proposed method.
  \end{abstract}


% {\bf AMS(MOS) Subject Classification:} 35Q10,
%65N30, 76D05 \vspace{0.5cm}
\section{Introduction}
The fluid flow between porous media and free-flow zones has extensive applications in hydrology, environmental science, and biofluid dynamics. A lot of researchers derive suitable mathematical and numerical models for fluid movement. The system can be viewed as a coupled problem with two physical systems interacting across an interface. The simplest mathematical formulation for the coupled problem is coupling of the Stokes and Darcy flow with proper interface conditions. The most suitable and popular interface conditions are called Beavers-Joseph-Saffman conditions \cite{P. G. Saffman}. However, Darcy's law only provides a linear relationship between the gradient of pressure and velocity in the coupled model, which usually fails for complex physical problems. Forchheimer \cite{P. Forchheimer} conducted flow experiments in sand packs and recognized that for moderate Reynolds numbers ($Re > 0.1$ approximately), Darcy's law is not adequate. He found that the pressure gradient and Darcy velocity should satisfy the Darcy-Forchheimer law. Since the great attention has received in the coupled model, a large number of traditional methods have been devoted to the coupled Stokes and Darcy flows problems \cite{E. J. Park1, M. Y. Kim, E. J. Park2, M. Discacciati, W. J. Layton, B. Riviere1, B. Riviere2, E. Burman, G. N. Gatica, V. Girault, K. Lipnikov}. However, the difficulty of the complicated high dimensional coupled problems causes the limitation of traditional methods.

Owing to the enormous potential in approximating high-dimensional nonlinear maps \cite{N. E. Cotter, K. Hornik1, K. Hornik2, K. Hornik3, G. Cybenko,K. He2,M. Telgrasky, H.M.Q. Liao}, deep learning has attracted growing attention in many applications, such as image, speech, text recognition and scientific computing \cite{A. Krizhevsky, G. Hinton, K. He1}. Many works have arisen based on the function approximation capabilities of the feed-forward fully-connected neural network to solve initial/boundary value problems \cite{X. Li, I. E. Lagaris1, I. E. Lagaris2, K. S. McFall} in the past decades. The solution to the system of equations can be obtained by minimizing the loss function, which typically consists of the residual error of the governing Partial Differential Equations (PDEs) along with initial/boundary values. Resently, Raissi \emph{etc.}\cite{Raissi1,Raissi2,Raissi3} developed Physics Informed Neural Networks (PINNs) \cite{Yang, Rao, Olivier, Lu, Fang, Pang}. Moreover, Sirignano and Spiliopoulos proposed the Deep learning Galerkin Method \cite{J. Sirignano} for solving high dimensional PDEs. Additionally, some recent works have successfully solved the second-order linear elliptic equations and the high dimensional Stokes problems \cite{Y. Khoo, J. Li1, J. Li2, J. Li3}. Although several excellent works have been performed in applying deep learning to solve PDEs, the topic for solving complicated coupled interface problems remains to be investigated.

Considering the performance of deep learning for solving PDEs, our contribution is to design the CDNNs as an efficient alternative model for complicated coupled physical problems. We can encode any underlying physical laws naturally as prior information to obey the law of physics. To satisfy the differential operators, boundary conditions and divergence conditions, we train the neural networks on batches of randomly sampled points. The method only inputs random sampling spatial coordinates without considering the nature of samples. Notably, we take the interface conditions as the constraint for the CDNNs. The approach is parallel and solves multiple variables independently at the same time. Specially, the optimal solution can be obtained by using the appropriate optimization method instead of a linear combination of basic functions. Further more, we validate the convergence of the loss function under certain conditions and the convergence of the CDNNs to the exact solution. Several numerical experiments are conducted to investigate the performance of the CDNNs.

The article is organized as follows: Section 2 introduces the coupled model and the relation methodology. Section 3 discusses the convergence of the loss function $J(\overline{\mathbf{U}})$ and the convergence of the CDNNs to the exact solution. Section 4 reveals some numerical experiments to illustrate the efficiency of the CDNNs. The article ends with conclusion in section 5.
\section{Methodology}
Let $\Omega_S$ and $\Omega_D$ be two bounded and simply connected polygonal domains in $\mathbb{R}^2$ such that $\partial\Omega_S\cap\partial\Omega_D=\Gamma\neq\emptyset$ and $\Omega_S\cap\Omega_D=\emptyset$. Then, let $\Gamma_S:=\partial\Omega_S\setminus\Gamma, \Gamma_D:=\partial\Omega_D\setminus\Gamma$ and $\bf{n}_S$ as the unit normal vector pointing from $\Omega_S$ to $\Omega_D$, $\mathbf{n}_D$ as the unit normal vector pointing from $\Omega_D$ to $\Omega_S$, on the interface $\Gamma$ we have $\mathbf{n}_D=-\mathbf{n}_S$. In addition, $\mathbf{t}$ represents the unit tangential vector along the interface $\Gamma$. Figure \ref{Fig.1} gives a schematic representation of the geometry.

When kinematic effects surpass viscous effects in a porous medium, the Darcy velocity $\mathbf{u}_D$ and the pressure gradient $\nabla{p}_D$ does not satisfy a linear relation. Instead, a nonlinear approximation, known as the Darcy-Forchheimermodel, is considered. When it is imposed on the porous medium $\Omega_D$ with homogeneous Dircihlet boundary condition on $\Gamma_D$ the equations read:
\begin{align}
\nabla\cdot \mathbf{u}_D&=f_D,~~~in~\Omega_D,\label{SDF-1}\\
\frac{\mu}{\rho}\mathbf{K}^{-1}\mathbf{u}_D+\frac{\beta}{\rho}\mid\mathbf{u}_D\mid\mathbf{u}_D+\nabla p_D&=\mathbf{g}_D,~~~in~\Omega_D,\label{SDF-2}\\
p_D&=0, ~~~on~\partial\Omega_D\setminus\Gamma,\label{SDF-3}
\end{align}
where $\mathbf{K}$ is the permeability tensor, assumed to be uniformly positive definite and bounded, $\rho$ is the density of the fluid, $\mu$ is its viscosity and $\beta$ is a dynamic viscosity, all assumed to be positive constants. In addition, $\mathbf{g}_D$ and $f_D$ are source terms. We remark that in this context we exploit homogeneous Dirichlet boundary condition, in fact, we can also consider homogeneous Neumann boundary condition, i.e., $\mathbf{u}_D\cdot\mathbf{n}_D=0$ on $\Gamma_D$ and the arguments used in this paper are still true.

The fluid motion in $\Omega_S$ is described by the Stokes equations:
\begin{align}
-\nu\Delta\mathbf{u}_S+\nabla p_S&=\mathbf{f}_S,~~~in~\Omega_S,\label{SDF-4}\\
\nabla\cdot\mathbf{u}_S&=0, ~~~on~\Omega_S,\label{SDF-5}\\
\mathbf{u}_S&=0, ~~~on~\partial\Omega_S\setminus\Gamma,\label{SDF-6}
\end{align}
where $\nu>0$ denotes the viscosity of the fluid.

On the interface, we prescribe the following interface conditions
\begin{align}
\mathbf{u}_S\cdot\mathbf{n}_S&=\mathbf{u}_D\cdot\mathbf{n}_S,~~~on~\Gamma,\label{SDF-7}\\
p_S-\nu\mathbf{n}_S\frac{\partial\mathbf{u}_S}{\partial\mathbf{n}_S}&=p_D,~~~on~\Gamma,\label{SDF-8}\\
-\nu \mathbf{t}\frac{\partial\mathbf{u}_S}{\partial\mathbf{n}_S}&=G\mathbf{u}_S\cdot\mathbf{t}, ~~~on~\Gamma.\label{SDF-9}
\end{align}
Condition (\ref{SDF-7}) represents continuity of the fluid velocity's normal components, (\ref{SDF-8}) represents the balance of forces acting across the interface, and (\ref{SDF-9}) is the Beaver-Joseph-Saffman condition \cite{G. S. Beaver}. The constant $G>0$ is given and is usually obtained from experimental data.
\begin{figure}
  \centering
  % Requires \usepackage{graphicx}
  \includegraphics[scale=0.5]{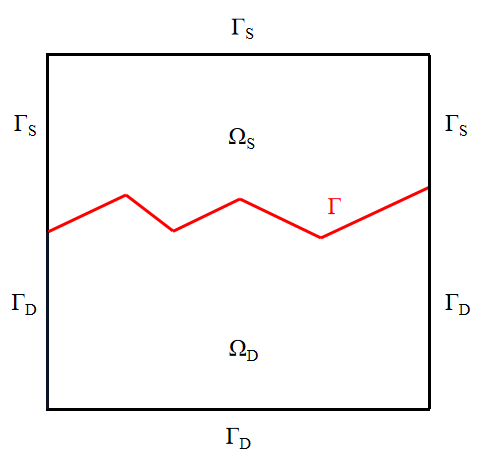}\\
  \caption{Coupled domain with interface $\Gamma$.}\label{Fig.1}
\end{figure}

For notational brevity, we set $\overline{\mathbf{u}}=(\mathbf{u}_S, \mathbf{u}_D, p_S,p_D)$ and recall the classical Sobolev spaces
$$\mathbb{X}_S^0=\{v_S\in [H^1(\OM_S)]^d: v_S|_{\GM_S}=\mathbf{0}\},$$
 $$\mathbb{Y}_D=\{q_D\in [W^{1,3/2}(\OM_D)]^2: q_D|_{\GM_D}=0\},$$
$$\mathbb{X}_S=\{v_S\in \mathbb{X}_S^0: div~ v_S=0\},$$
where
$$H^k(\Omega)=\Big\{\upsilon\in L^2(\Omega): D_w^{\alpha}\upsilon\in L^2(\Omega), \forall\alpha: \mid\alpha\mid\leq k\Big\}, $$
and their norm
$$\parallel\upsilon\parallel_k=\sqrt{(\upsilon,\upsilon)_k}=\bigg\{\sum_{\mid\alpha\mid=0}^{k}\int_{\Omega}(D_w^{\alpha}\upsilon)^{2}dx\bigg\}^{\frac{1}{2}}, ~~\parallel\upsilon\parallel_{W_{k,p}}=\Big\{\sum_{|\alpha|\leq k}\|\upsilon\|^p_{L^p}\Big\}^{1/p}.$$

Particularly,
$$\|v\|_{k}=\|v\|_{W^{k,2}}$$
where $k>0$ is a positive integer and  $\parallel\upsilon\parallel_0$ denotes the norm on $L^2(\Omega)$ or $(L^2(\Omega))^2$, $D_w^{\alpha}\upsilon$ is the generalized derivative of $\upsilon$. Moreover, $(\cdot,\cdot)_D$ represents the inner product in the domain $D$ and $<\cdot,\cdot>$  represents the inner product on the interface $\Gamma$.
\begin{figure}[ht]
  \centering
  \includegraphics[scale=0.45]{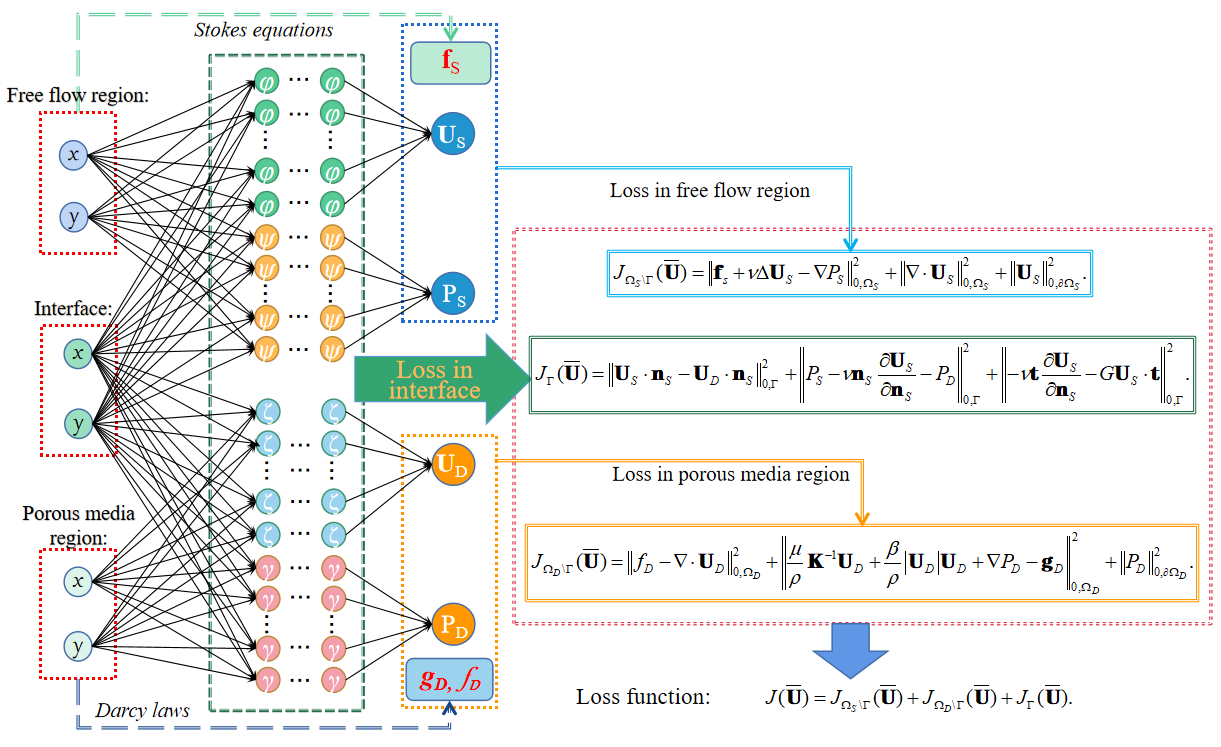}\\
  \caption{The structure of the CDNNs.}\label{nn-structure}
\end{figure}

To solve coupling of the Stokes and Darcy-Forchheimer problems, we propose the CDNNs in Figure \ref{nn-structure}. Further more, we give observations of the state variable $\overline{\mathbf{U}}(x;\theta)=\big(\mathbf{U}_S(x;\theta_{1})$, $\mathbf{U}_D(x;\theta_{2})$, $P_S(x;\theta_{3})$, $P_D(x;\theta_{4})\big)$, which is the neural network solution to the coupled Stokes and Darcy-Forchheimer problem (\ref{SDF-1})-(\ref{SDF-9}), $(\mathbf{\theta_{1}}, \mathbf{\theta_{3}})$ and $(\mathbf{\theta_{2}}, \mathbf{\theta_{4}})$ are the stacked parameters of $\theta$ for Stokes and Darcy respectively. The following constrained optimization procedure aims to reconstruct the parameters $\theta$ by minimizing the loss function
\begin{equation}\label{J}
J[\overline{\mathbf{U}}]=J_{\Omega_S\setminus\Gamma}[\overline{\mathbf{U}}]+J_{\Omega_D\setminus\Gamma}[\overline{\mathbf{U}}]+J_\Gamma[\overline{\mathbf{U}}].
\end{equation}
where
\begin{equation}\label{J123}
\begin{split}
J_{\Omega_S\setminus\Gamma}(\overline{\mathbf{U}})=&\Big\|\mathbf{f}_S+\nu\Delta\mathbf{U}_S(x;\theta_1)-\nabla P_S(x;\theta_3)\Big\|_{0, \Omega_S, \omega_{1}}^{2}\\
&+\Big\|\nabla\cdot {\mathbf{U_S}}(x;\theta_1)\Big\|_{0, \Omega_S,\omega_{1}}^{2}
+\Big\|\mathbf{U_S}(x;\theta_1)\Big\|_{0, \partial\Omega_S\setminus\Gamma, \omega_{2}}^{2}\\
J_{\Omega_D\setminus\Gamma}(\overline{\mathbf{U}})=&\Big\|f_D-\nabla\cdot \mathbf{U}_D(x;\theta_2)\Big\|_{0, \Omega_D, \omega_{1}}^{2}\\
&+\Big\|\frac{\mu}{\rho}\mathbf{K}^{-1}\mathbf{U}_D(x;\theta_2)+\frac{\beta}{\rho}\big|\mathbf{U}_D(x;\theta_2)\big|\mathbf{U}_D(x;\theta_2)
+\nabla P_D(x;\theta_4)-\mathbf{g}_D\Big\|_{0, \Omega_D,\omega_{1}}^{2}\\
&+\Big\|P_D(x;\theta_4)\Big\|_{0, \partial\Omega_D\setminus\Gamma, \omega_{2}}^{2}\\
J_{\Gamma}(\overline{\mathbf{U}})=&\Big\|\mathbf{U}_S(x;\theta_1)\cdot\mathbf{n}_S-\mathbf{U}_D(x;\theta_2)\cdot\mathbf{n}_S\Big\|_{0, \Gamma, \omega_{3}}^{2}\\
&+\Big\|P_S(x;\theta_3)-\nu\mathbf{n}_S\frac{\partial\mathbf{U}_S(x;\theta_1)}{\partial\mathbf{n}_S}-P_D(x;\theta_4)\Big\|_{0, \Gamma,\omega_{3}}^{2}\\
&+\Big\|-\nu\mathbf{t}\frac{\partial\mathbf{U}_S(x;\theta_1)}{\partial\mathbf{n}_S}-G\mathbf{U}_S(x;\theta_1)\cdot\mathbf{t}\Big\|_{0, \Gamma, \omega_{3}}^{2}\\
\end{split}
\end{equation}

The nodal values of the parameters in the input layer admitted by the deep learning model. Furthermore, it should be noted that $J(\overline{\mathbf{U}})$ can measure how well the approximate solution satisfies differential operators, divergence conditions, boundary conditions and interface conditions. Notice that
$$\big\| f(y)\big\|_{0,\mathcal{Y},\omega}=\int_{\mathcal{Y}}\big| f(y)\big|^{2}\omega(y)dy,$$
where $\omega(y)$ is the probability density of $y$ in $\mathcal{Y}$. Especially, if $J(\overline{\mathbf{U}})=0$ then $\overline{\mathbf{U}}$ is the solution to the coupled Stokes and Darcy-Forchheimer problems (\ref{SDF-1})-(\ref{SDF-9}). Due to the infeasibility to estimate $\theta$ by directly minimizing $J(\overline{\mathbf{U}})$ when integrated over a higher dimensional region, so we apply a sequence of randomly sampled points from domain instead forming mesh grid. The main steps of the CDNNs for the coupled Stokes and Darcy-Forchheimer equations are presented as \textbf{Algorithm} \ref{alg}.
\begin{algorithm}
\caption{The CDNNs for the coupled problems}\label{alg}
\KwIn{$\rho_{\omega_1}^{(n)}=\big\{ x_S^{n},x_D^{n}\big\}, \rho_{\omega_2}^{(n)}=\big\{ r_S^{n}, r_D^{n}\big\},  \rho_{\omega_3}^{(n)}=\big\{r_{\Gamma}^{n})\big\}$, max iterations $M$,  learning rate $\alpha$ }
\KwOut{$\theta_{n+1}$}
1. Randomly generated sample points $\rho^{(n)}=\big\{\rho_{\omega_1}^{(n)},  \rho_{\omega_2}^{(n)},  \rho_{\omega_3}^{(n)}\big\}$ from $(\Omega_S, \Omega_D)$, $(\partial\Omega_S\backslash\Gamma, \partial\Omega_D\backslash\Gamma)$ and $\Gamma$ by the respective probability densities $\omega_{1}$, $\omega_{2}$ and $\omega_{3}$; \\
2. Initialize the parameters  $\theta$\;
3.
\While{iterations $\leq M$}{
read current;
      \begin{align*}
    G(\rho^{(n)}, \theta^{n})&= G_S(\rho^{(n)}, \theta^{n})+ G_D(\rho^{(n)}, \theta^{n})+ G_{\Gamma}(\rho^{(n)}, \theta^{n}),
  \end{align*}\\
where
        \begin{align*}
    G_S(\rho^{(n)}, \theta^{n})&=\Big(\mathbf{f}_S+\nu\Delta\mathbf{U}_S(x_S^n;\theta_1)-\nabla P_S(x_S^n;\theta_3)\Big)^{2}\\
    &+\Big(\nabla\cdot {\mathbf{U_S}}(x_S^n;\theta_1)\Big)^{2}
+\Big(\mathbf{U_S}(r_S^n;\theta_1)\Big)^{2},\\
    G_D(\rho^{(n)}, \theta^{n})&=\Big(\frac{\mu}{\rho}K^{-1}\mathbf{U}_D(x_D^{n};\theta_2)+\frac{\beta}{\rho}\big|\mathbf{U}_D(x_D^{n};\theta_2)\big|\mathbf{U}_D(x_D^{n};\theta_2)
+\nabla P_D(x_D^{n};\theta_4)-\mathbf{g}_D\Big)^{2}\\
&+\Big(f_D-\nabla\cdot \mathbf{U}_D(x_D^{n};\theta_2)\Big)^{2}+\Big(P_D(r_D^{n};\theta_4)\Big)^{2},
  \end{align*}\\
and
        \begin{align*}
    G_{\Gamma}(\rho^{(n)}, \theta^{n})&=\Big(\mathbf{U}_S(x;\theta_1)\cdot\mathbf{n}_S-\mathbf{U}_D(x;\theta_2)\cdot\mathbf{n}_S\Big)^{2}\\
&+\Big(P_S(x;\theta_3)-\nu\mathbf{n}_S\frac{\partial\mathbf{U}_S(x;\theta_1)}{\mathbf{n}_S}-P_D(x;\theta_4)\Big)^{2}\\
&+\Big(-\nu\mathbf{t}\frac{\partial\mathbf{U}_S(x;\theta_1)}{\partial\mathbf{n}_S}-G\mathbf{U}_S(x;\theta_1)\cdot\mathbf{t}\Big)^{2},
  \end{align*}\\
 \textbf{ and}\\
   $$\theta^{n+1}=\theta_{n}-\alpha\nabla _{\theta}G(\rho^{(n)},\theta^{n}).$$\\
    \eIf{$\underset{n\rightarrow\infty}{\lim}\| \nabla_{\theta}G(\rho^{(n)},\theta_{n})\|=0$}{
        return the parameters $\theta_{n+1}$\;
    }{
        go back to the beginning of current section.\
    }
}
\end{algorithm}
Another noticeable point is that the term $\nabla_{\theta}G(\theta^{n},z^{(n)})$ is unbiased estimate of $\nabla_{\theta}J\big(\overline{\mathbf{U}}(\cdot;\theta^{n})\big)$ because the population parameters can be estimated by sample mathematical expectations.

\section{Convergence}
\vspace{0.3cm} According to the definition of loss function $J(\overline{\mathbf{U}})$, it can measure how well $\overline{\mathbf{U}}$ satisfies the equations (\ref{SDF-1})-(\ref{SDF-9}). Neural networks are a set of algorithms for classification and regression tasks inspired by the biological neural networks in brains. There have various types of neural networks with different neuron connection forms and architectures. According to the \cite{K. Hornik3}, if there is only one hidden layer and output, the set of functions implemented by following networks with $m_1, m_2, m_3$ and $m_4$ hidden units for coupling of the Stokes and Darcy-Forchheimer problems are
\begin{equation}\nonumber
\begin{split}
[\mathfrak{C}_{\mathbf{U}_S}^{\mathrm{m_1}}(\varphi)]^{d}&=\Big\{\Theta(x): \mathbb{R^{\mathrm{d}}\mapsto\mathbb{R^{\mathrm{d}}}\Big|}\Theta(x)={\sum\limits_{i=1}^{m_1}}\beta_{i}\varphi\big({\sum\limits_{j=1}^{d}}\sigma_{j,i}x_{j}+c_{i}\big)\Big\},\\
[\mathfrak{C}_{\mathbf{U}_D}^{\mathrm{m_2}}(\zeta)]^{d}&=\Big\{\Lambda(x): \mathbb{R^{\mathrm{d}}\mapsto\mathbb{R^{\mathrm{d}}}\Big|}\Lambda(x)={\sum\limits_{i=1}^{m_2}}\beta'_{i}\zeta\big({\sum\limits_{j=1}^{d}}\sigma'_{j,i}x_{j}+c'_{i}\big)\Big\},\\
\mathfrak{C}_{P_S}^{\mathrm{m_3}}(\psi)&=\Big\{\Psi(t,x): \mathbb{R^{\mathrm{d}}\mapsto\mathbb{R}\Big|}\Psi(x)={\sum\limits_{i=1}^{m_3}}\beta''_{i}\psi\big({\sum\limits_{j=1}^{d}}\sigma''_{j,i}x_{j}+c''_{i}\big)\Big\},\\
\mathfrak{C}_{P_D}^{\mathrm{m_4}}(\gamma)&=
\Big\{\Upsilon(t,x): \mathbb{R^{\mathrm{d}}\mapsto\mathbb{R}\Big|}\Upsilon(x)={\sum\limits_{i=1}^{m_4}}\beta'''_{i}\gamma\big({\sum\limits_{j=1}^{d}}\sigma'''_{j,i}x_{j}+c'''_{i}\big)\Big\},\\
\end{split}
\end{equation}
where $\Theta(x)=\big(\Theta_1(x),\Theta_2(x),\cdots ,\Theta_d(x)\big), \Lambda(x)=\big(\Lambda_1(x),\Lambda_2(x),\cdots ,\Lambda_d(x)\big)$, $\varphi$, $\zeta$, $\psi$ and $\gamma$ are the shared activation functions of the hidden units in $\mathcal{C}^2(\Omega)$, bounded and non-constant. $x_j$ is input, $\beta_i,\beta'_i,\beta''_i,\beta'''_i,\sigma_{ji},\sigma'_{ji}, \sigma''_{ji} $and $\sigma'''_{ji}$ are weights, $c_i,c'_i,c''_i$ and $c'''_i$ are thresholds of the neural networks.

More generally, we use the similar notation
$$[\mathfrak{C}_{\mathbf{U}_S}(\varphi)]^{d}\times[\mathfrak{C}_{\mathbf{U}_D}(\zeta)]^{d}\times\mathfrak{C}_{P_S}(\psi)\times\mathfrak{C}_{P_D}(\gamma)$$
for the multi layer neural networks with arbitrarily large number of hidden units $m_1, m_2, m_3$ and $m_4$. In particular, the parameters and the activation function in each dimension of $[\mathfrak{C}_{\mathbf{U}_S}^{m_1}(\varphi)]^{d}$ or $[\mathfrak{C}_{\mathbf{U}_D}^{\mathrm{m_2}}(\zeta)]^{d}$ are same as before, we set $n$ as the number of the neurons in numerical experiments. Then the parameters of the CDNNs can be formalized as follows
\begin{align*}
\theta_{1}^{k}&=(\beta_{1},\cdots,\beta_{n},\sigma_{11},\cdots,\sigma_{dn},c_{1},\cdots,c_{n}),\\
\theta_{2}^{k}&=(\beta'_{1},\cdots,\beta_{n},\sigma'_{11},\cdots,\sigma'_{dn},c'_{1},\cdots,c'_{n}),\\
\theta_{3}&=(\beta''_{1},\cdots,\beta'_{n},\sigma''_{11},\cdots,\sigma''_{dn},c''_{1},\cdots,c''_{n}),\\
\theta_{4}&=(\beta'''_{1},\cdots,\beta''_{n},\sigma'''_{11},\cdots,\sigma'''_{dn},c'''_{1},\cdots,c'''_{n}),
\end{align*}
where $k=1,2,\ldots,d, \theta_{1}\in \mathbb{R}^{(2+d)nd},  \theta_{2}\in \mathbb{R}^{(2+d)nd}, \theta_{3}\in \mathbb{R}^{(2+d)n}$ and $\theta_{4}\in \mathbb{R}^{(2+d)n}$.

In the next two subsections, we prove that the neural network $\overline{\mathbf{U}}^n$ with $n$ hidden units for $\mathbf{U}_S^n, \mathbf{U}_D^n, P_S^n$ and $ P_D^n$ satisfy the differential operators, boundary conditions, divergence conditions and interface conditions arbitrarily well for sufficiently large $n$. More importantly, we confirm that there exists $\overline{\mathbf{U}}^n\in[\mathfrak{C}_{\mathbf{U}_S}(\varphi)]^{d}\times[\mathfrak{C}_{\mathbf{U}_D}(\zeta)]^{d}\times\mathfrak{C}_{P_S}(\psi)\times\mathfrak{C}_{P_D}(\gamma)$ such that $J(\overline{\mathbf{U}}^n)\rightarrow0$ as $n\rightarrow\infty$. Another significant consideration, we give the convergence of $\overline{\mathbf{U}}^{n}\rightarrow\overline{\mathbf{u}}$ as $n\rightarrow\infty$ where $\overline{\mathbf{u}}$ is the exact solution to the coupled equations (\ref{SDF-1})-(\ref{SDF-9}).

\subsection{Convergence of the loss function $J(\overline{\mathbf{U}})$}
In this subsection, we prove the CDNNs $\overline{\mathbf{U}}$ can make the loss function $J(\overline{\mathbf{U}})$ arbitrarily small.
\newtheorem{assumption}{Assumption}[section]
\begin{assumption}\label{assumption1}
$\nabla \mathbf{v}(x),\triangle \mathbf{v}(x)$ and $\nabla q(x)$ are locally Lipschitz with Lipschitz coefficient that they have at most polynomial growth on $\mathbf{v}(x)$ and $q(x)$. Then, for some constants $0\leq q_{i}\leq\infty(i=1,2,3,4)$ we have
\begin{equation}\label{1}
|\triangle \mathbf{V}(x;\theta)-\triangle \mathbf{v}(x)|\leq(|\nabla \mathbf{V}(x;\theta)|^{q_{1}/2}+|\nabla \mathbf{v}(x)|^{q_{2}/2})|\nabla \mathbf{V}(x;\theta)-\nabla v(x)|,
\end{equation}
\begin{equation}\label{2}
|\nabla Q(x;\theta)-\nabla q(x)|\leq(| Q(x;\theta)|^{q_{3}/2}+| q(x)|^{q_{4}/2})| Q(x;\theta)-q(x)|,
\end{equation}
\begin{equation}\label{3}
|\nabla \mathbf{V}(x;\theta)-\nabla \mathbf{v}(x)|\leq(| \mathbf{V}(x;\theta)|^{q_{5}/2}+| \mathbf{v}(x)|^{q_{6}/2})| \mathbf{V}(x;\theta)-\mathbf{v}(x)|.
\end{equation}
\end{assumption}

\newtheorem{thm}{\bf Theorem}[section]
\begin{thm}\label{thm1}
Under the Assumption \ref{assumption1}, there exists a neural network $\overline{\mathbf{U}}\in[\mathfrak{C}_{\mathbf{U}_S}(\varphi)]^{d}\times[\mathfrak{C}_{\mathbf{U}_D}(\zeta)]^{d}\times\mathfrak{C}_{P_ S}(\psi)\times\mathfrak{C}_{P_D}(\gamma)$, satisfying
\begin{equation}\label{4}
J(\overline{\mathbf{U}})\leq C\epsilon^2, \ \forall \epsilon>0,
\end{equation}
where $C$ depends on the data $\{\Omega_S,\Omega_D,\Gamma, \mu, \rho, \beta, K^{-1}, \omega_{1}, \omega_{2}, \omega_{3}, f_D, \mathbf{g}_D, \mathbf{f}_S\}$.
\end{thm}
\begin{proof}
From Theorem 3 of \cite{K. Hornik3}, we obtain that there exists $\overline{\mathbf{U}}\in[\mathfrak{C}_{\mathbf{U}_S}(\varphi)]^{d}\times[\mathfrak{C}_{\mathbf{U}_D}(\zeta)]^{d}\times\mathfrak{C}_{P_ S}(\psi)\times\mathfrak{C}_{P_D}(\gamma)$ which are uniformly 2-dense on compacts of $\mathcal{C} ^{2}(\bar{\Omega}_S)\times\mathcal{C} ^{2}(\bar{\Omega}_D)\times\mathcal{C} ^{1}(\bar{\Omega}_S)\times\mathcal{C} ^{1}(\bar{\Omega}_D)$. It means that for $\overline{\mathbf{u}}(x)\in \mathcal{C} ^{2}(\bar{\Omega}_S)\times\mathcal{C} ^{2}(\bar{\Omega}_D)\times\mathcal{C} ^{1}(\bar{\Omega}_S)\times\mathcal{C} ^{1}(\bar{\Omega}_D), \forall\epsilon >0$, we confirm that
\begin{equation}\label{5}
\underset{a\leq2}{max}\underset{x\in\Omega_S}{sup}| \partial_{x}^{a}\mathbf{U}_S(x; \theta_1)-\partial_{x}^{a}\mathbf{u}_S(x)|<\epsilon,
\end{equation}
\begin{equation}\label{6}
\underset{a\leq2}{max}\underset{x\in\Omega_D}{sup}| \partial_{x}^{a}\mathbf{U}_D(x; \theta_2)-\partial_{x}^{a}\mathbf{u}_D(x)|<\epsilon,
\end{equation}
\begin{equation}\label{7}
\underset{x\in\Omega_S}{sup}| P_ S(x;\theta_3)-p_ S(x)|<\epsilon.
\end{equation}
\begin{equation}\label{8}
\underset{x\in\Omega_D}{sup}| P_D(x;\theta_4)-p_D(x)|<\epsilon.
\end{equation}
Firstly, we recall the form and discuss the convergence of $J_{\Omega_S\setminus\Gamma}(\overline{\mathbf{U}})$,
\begin{equation}\label{9}
\begin{split}
J_{\Omega_S\setminus\Gamma}(\overline{\mathbf{U}})&=\Big\|\mathbf{f}_S+\nu\Delta\mathbf{U}_S(x;\theta_1)-\nabla P_ S(x;\theta_3)\Big\|_{0, \Omega_S, \omega_{1}}^{2}\\
&+\Big\|\nabla\cdot {\mathbf{U_S}}(x;\theta_1)\Big\|_{0, \Omega_S,\omega_{1}}^{2}
+\Big\|\mathbf{U_S}(x;\theta_1)\Big\|_{0, \partial\Omega_S\setminus\Gamma, \omega_{2}}^{2}\\
\end{split}
\end{equation}
According to the Assumption \ref{assumption1}, by using the \emph{H$\ddot{o}$lder} inequality and\emph{ Young} inequality, setting conjugate numbers $r_{1}$ and $r_{2}$ such that $\frac{1}{r_{1}}+\frac{1}{r_{2}}=1$, it follows that
\begin{equation}
\begin{aligned}\label{10}
&\int_{\Omega_S}|\triangle \mathbf{U}(x;\theta_1)-\triangle \mathbf{u}(x)|^{2}d\omega_{1}(x) \\
\leq  &\int_{\Omega_S}\Big(|\nabla \mathbf{U}(x;\theta_1)|^{l_{1}}+|\nabla \mathbf{u}(x)|^{l_{2}}\Big)\Big(\nabla \mathbf{U}(x;\theta_1)-\nabla \mathbf{u}(x)\Big)^{2}d\omega_{1}(x)\\
\leq  &\Big[\int_{\Omega_S}\Big(|\nabla \mathbf{U}(x;\theta_1)|^{l_{1}}+|\nabla \mathbf{u}(x)|^{l_{2}}\Big)^{r_{1}}d\omega_{1}(x)\Big]^{1/r_{1}}\\
&\times\Big[\int_{\Omega_S}\Big(\nabla \mathbf{U}(x;\theta_1)-\nabla \mathbf{u}(x)\Big)^{2r_{2}}d\omega_{1}(x)\Big]^{1/r_{2}} \\
\leq  &\Big[\int_{\Omega_S}\Big(|\nabla \mathbf{U}(x;\theta_1)-\nabla \mathbf{u}(x)|^{l_{1}}+|\nabla \mathbf{u}(x)|^{l_{1}\vee l_{2}}\Big)^{r_{1}}d\omega_{1}(x)\Big]^{1/r_{1}}\\
&\times\Big[\int_{\Omega_S}\Big(\nabla \mathbf{U}(x;\theta_1)-\nabla \mathbf{u}(x)\Big)^{2r_{2}}d\omega_{1}(x)\Big]^{1/r_{2}}\\
\leq  &C \Big(\epsilon^{l_1}+\underset{{\Omega}_S}{sup}| \nabla \mathbf{u}(x) |^{l_1\vee l_2})\Big)\epsilon^{2} ,
\end{aligned}
\end{equation}
here we set $l_{1}\vee l_{2}=max\{l_{1}, l_{2}\}$.
In the same way,
\begin{equation}
\begin{aligned}\label{11}
&\int_{\Omega_S}\Big\vert\nabla P_ S(x;\theta_3)-\nabla p_ S(x)\Big\vert^{2}d\omega_{1}(x) \\
\leq  &\int_{\Omega_S}\Big(| P_ S(x;\theta_3)|^{l_{3}}+| p_ S(x)|^{l_{4}}\Big)\Big(P_ S(x;\theta_3)-p_ S(x)\Big)^{2}d\omega_{1}(x)\\
\leq  &\Big[\int_{\Omega_S}\Big(| P_ S(x;\theta_3)|^{l_{3}}+| p_ S|^{l_{4}}\Big)^{r_{3}}d\omega_{1}(x)\Big]^{1/r_{3}}\times\Big[\int_{\Omega_S}
\Big(P_ S(x;\theta_3)-p_ S(x)\Big)^{2r_{4}}d\omega_{1}(x)\Big]^{1/r_{4}}\\
\leq  &\Big[\int_{\Omega_S}\Big(| P_ S(x;\theta_3)-p_ S(x)|^{l_{3}}+| p_ S(x)|^{l_{3}\vee l_{4}}\Big)^{r_{3}}d\omega_{1}(x)\Big]^{1/r_{3}}\times\Big[\int_{\Omega_S}
\Big(P_ S(x;\theta_3)-p_ S(x)\Big)^{2r_{4}}d\omega_{1}(x)\Big]^{1/r_{4}}\\
\leq &C \Big(\epsilon^{l_3}+\underset{{\Omega}_S}{sup}|  p_ S(x) |^{l_3\vee l_4})\Big)\epsilon^{2},
\end{aligned}
\end{equation}
where $\frac{1}{r_{3}}+\frac{1}{r_{4}}=1$ and $l_{3}\vee l_{4}=max\{l_{3}, l_{4}\}$.

For the boundary condition, we have
\begin{align}\label{12}
&\int_{\partial\Omega_S\setminus\Gamma}| U_S(x;\theta_1)-u_S(x)|^{2}d\omega_{2}(x)\leq C\epsilon^{2}.
\end{align}
Owing to (\ref{10}) - (\ref{12}), we can conclude that
\begin{equation}
\begin{aligned}\label{13}
J_{\Omega_S\setminus\Gamma}(\overline{\mathbf{U}})=&\Big\|\mathbf{f}_S+\nu\Delta\mathbf{U}_S(x;\theta_1)-\nabla P_ S(x;\theta_3)\Big\|_{0, \Omega_S, \omega_{1}}^{2}\\
&+\Big\|\nabla\cdot {\mathbf{U_S}}(x;\theta_1)\Big\|_{0, \Omega_S,\omega_{1}}^{2}
+\Big\|\mathbf{U_S}(x;\theta_1)\Big\|_{0, \partial\Omega_S\setminus\Gamma, \omega_{2}}^{2}\\
\leq&\nu\big\|\Delta\mathbf{U}_S(x;\theta_1)-\Delta\mathbf{u}_S(x)\big\|_{0,\Omega_S,\omega_{1}}^{2}+\big\|\nabla {P_ S}(x;\theta_3)-\nabla p_ S(x)\big\|_{0,\Omega_S,\omega_{1}}^{2}\\
&+\big\|\nabla\cdot {\mathbf{U}_S}(x;\theta_1)\big\|_{0,\Omega_S,\omega_{1}}^{2}+\big\|\mathbf{U}(x;\theta_1)\big\|_{0,\partial\Omega_S\setminus\Gamma,\omega_{2}}^{2}\\
\leq&\nu\int_{\Omega_S}|\triangle \mathbf{U}_S(x;\theta_1)-\triangle \mathbf{u}_S(x)|^{2}d\omega_{1}(x)+\int_{\Omega_S}|\nabla P_ S(x;\theta_3)-\nabla p_ S(x)|^{2}d\omega_{1}(x)\\
&+\int_{\Omega_S}|\nabla\cdot \mathbf{u}_S(x)|^{2}d\omega_{1}(x)+\int_{\Omega_S}\Big|\nabla\cdot\Big(\mathbf{U}_S(x;\theta_1)-\mathbf{u}_S(x)\Big)\Big|^{2}d\omega_{1}(x)\\
&+\int_{\partial\Omega_S\setminus\Gamma}| \mathbf{U}_S(x;\theta_1)-\mathbf{u}_S(x)|^{2}d\omega_{2}(x)\\
\leq  &C\epsilon^{2}.
\end{aligned}
\end{equation}

Next, we remain to prove the convergence of $J_{\Omega_D\setminus\Gamma}(\overline{\mathbf{U}})$ and $J_{\Gamma}(\overline{\mathbf{U}})$. We know that
\begin{equation}\label{14}
\begin{split}
J_{\Omega_D\setminus\Gamma}(\overline{\mathbf{U}})=&\Big\|f_D-\nabla\cdot \mathbf{U}_D(x;\theta_2)\Big\|_{0, \Omega_D, \omega_{1}}^{2}+\Big\|P_D(x;\theta_4)\Big\|_{0, \partial\Omega_D\setminus\Gamma, \omega_{2}}^{2}\\
&+\Big\|\frac{\mu}{\rho}K^{-1}\mathbf{U}_D(x;\theta_2)+\frac{\beta}{\rho}\big|\mathbf{U}_D(x;\theta_2)\big|\mathbf{U}_D(x;\theta_2)
+\nabla P_D(x;\theta_4)-\mathbf{g}_D\Big\|_{0, \Omega_D,\omega_{1}}^{2}.
\end{split}
\end{equation}
From (\ref{8}), we have
\begin{equation}\label{15}
\begin{aligned}
&\int_{\Omega_D}\Big\vert\nabla P_D(x;\theta_4)-\nabla p_D(x)\Big\vert^{2}d\omega_{1}(x) \\
\leq  &\int_{\Omega_D}\Big(| P_D(x;\theta_4)|^{l_{5}}+| p_D(x)|^{l_{6}}\Big)\Big(P_D(x;\theta_4)-p_D(x)\Big)^{2}d\omega_{1}(x),
\end{aligned}
\end{equation}
which can be updated by using the H$\ddot{o}$lder inequality and Young inequality, thus we have
\begin{equation}\label{15-2}
\begin{aligned}
&\Big[\int_{\Omega_D}\Big(| P_D(x;\theta_4)|^{l_{5}}+| p_D|^{l_{6}}\Big)^{r_{5}}d\omega_{1}(x)\Big]^{1/r_{5}}\\
&\times\Big[\int_{\Omega_D}
(P_D(x;\theta_4)-p_D(x))|^{2r_{6}}d\omega_{1}(x)\Big]^{1/r_{6}}\\
\leq  &\Big[\int_{\Omega_D}\Big(| P_D(x;\theta_4)-p_D(x)|^{l_{5}}+| p_D(x)|^{l_{5}\vee l_{6}}\Big)^{r_{5}}d\omega_{1}(x)\Big]^{1/r_{5}}\\
&\times\Big[\int_{\Omega_D}
\Big(P_D(x;\theta_4)-p_D(x)\Big)^{2r_{6}}d\omega_{1}(x)\Big]^{1/r_{6}}\\
\leq &C \Big(\epsilon^{l_5}+\underset{{\Omega}_D}{sup}|  p_D(x) |^{l_5\vee l_6})\Big)\epsilon^{2},
\end{aligned}
\end{equation}
where $\frac{1}{r_{5}}+\frac{1}{r_{6}}=1$ and $l_{5}\vee l_{6}=max\{l_{5}, l_{6}\}$.

Next we prove the boundedness of term $|\mathbf{U}_D(x;\theta_2)\big|\mathbf{U}_D(x;\theta_2)$,
\begin{equation}\label{16-1}
\begin{split}
&\int_{\Omega_D}\Big(|\mathbf{U}_D(x;\theta_2)\big|\mathbf{U}_D(x;\theta_2)-|\mathbf{u}_D(x)\big|\mathbf{u}_D(x)\Big)^{2}d\omega_{1}(x)\\
=&\int_{\Omega_D}\Big(\mathbf{U}_D(x;\theta_2)\Big(|\mathbf{U}_D(x;\theta_2)|-|\mathbf{u}_D(x)|\Big)+|\mathbf{u}_D(x)|\Big(\mathbf{U}_D(x;\theta_2)-\mathbf{u}_D(x)\Big)\Big)^2d\omega_{1}(x)\\
=&\int_{\Omega_D}\Big(\mathbf{U}_D(x;\theta_2)\Big(|\mathbf{U}_D(x;\theta_2)|-|\mathbf{u}_D(x)|\Big)\Big)^2d\omega_{1}(x)+\int_{\Omega_D}\Big(|\mathbf{u}_D(x)|\Big(\mathbf{U}_D(x;\theta_2)-\mathbf{u}_D(x)\Big)\Big)^2d\omega_{1}(x)\\
&+2\int_{\Omega_D}\Big(|\mathbf{u}_D(x)|\mathbf{U}_D(x;\theta_2)\Big(|\mathbf{U}_D(x;\theta_2)|-|\mathbf{u}_D(x)|\Big)\Big(\mathbf{U}_D(x;\theta_2)-\mathbf{u}_D(x)\Big)d\omega_{1}(x),
\end{split}
\end{equation}
where
\begin{equation}\label{16-2}
\begin{aligned}
&\int_{\Omega_D}\Big(\mathbf{U}_D(x;\theta_2)\Big(|\mathbf{U}_D(x;\theta_2)|-|\mathbf{u}_D(x)|\Big)\Big)^2d\omega_{1}(x)\\
\leq&\Big[\int_{\Omega_D}\Big(\mathbf{U}_D(x;\theta_2)\Big)^{2r_7}d\omega_{1}(x)\Big]^{1/{r_7}}\times\Big[\int_{\Omega_D}\Big(|\mathbf{U}_D(x;\theta_2)|+|\mathbf{u}_D(x)|\Big)\Big)^{2r_8}d\omega_{1}(x)\Big]^{r_8}\\
\leq&\Big[\int_{\Omega_D}\Big(\Big(\mathbf{U}_D(x;\theta_2)-\mathbf{u}_D(x)\Big)+\mathbf{u}_D(x)\Big)^{2r_7}d\omega_{1}(x)\Big]^{1/{r_7}}\\
&\times\Big[\int_{\Omega_D}\Big(|\mathbf{U}_D(x;\theta_2)-\mathbf{u}_D(x)|+|\mathbf{u}_D(x)|\Big)^{2r_8}d\omega_{1}(x)\Big]^{r_8},
\end{aligned}
\end{equation}
by using the \emph{H$\ddot{o}$lder} inequality and \emph{Young} inequality, setting conjugate numbers $r_{7}$ and $r_{8}$ such that $\frac{1}{r_{7}}+\frac{1}{r_{8}}=1$.

Similarly,  we can obtain
\begin{equation}\label{16-3}
\begin{aligned}
&\int_{\Omega_D}\Big(|\mathbf{u}_D(x)|\Big(\mathbf{U}_D(x;\theta_2)-\mathbf{u}_D(x)\Big)\Big)^2d\omega_{1}(x)\\
\leq&[\int_{\Omega_D}\Big(|\mathbf{u}_D(x)|)^{r_9}]^{1/r_9}\times[\int_{\Omega_D}\Big(\mathbf{U}_D(x;\theta_2)-\mathbf{u}_D(x)\Big)\Big)^{2r_{10}}]^{1/r_{10}}d\omega_{1}(x),
\end{aligned}
\end{equation}
here $\frac{1}{r_{9}}+\frac{1}{r_{10}}=1$. Furthermore, we set $\frac{1}{r_{11}}+\frac{1}{r_{12}}=1$ and $ \frac{1}{r_{13}}+\frac{1}{r_{14}}=1,$
\begin{equation}\label{16-4}
\begin{aligned}
&\int_{\Omega_D}|\mathbf{u}_D(x)|\mathbf{U}_D(x;\theta_2)\Big(|\mathbf{U}_D(x;\theta_2)|-|\mathbf{u}_D(x)|\Big)\Big(\mathbf{U}_D(x;\theta_2)-\mathbf{u}_D(x)\Big)d\omega_{1}(x)\\
\leq&\Big[\int_{\Omega_D}\Big(\mathbf{U}_D(x;\theta_2)\Big(|\mathbf{U}_D(x;\theta_2)|-|\mathbf{u}_D(x)|\Big)\Big)^{r_{11}}d\omega_{1}(x)\Big]^{1/r_{11}}\\
&\times\Big[\int_{\Omega_D}\Big(|\mathbf{u}_D(x)|\Big(\mathbf{U}_D(x;\theta_2)-\mathbf{u}_D(x)\Big)\Big)^{r_{12}}d\omega_{1}(x)\Big]^{1/r_{12}}\\
\leq&\Big[\int_{\Omega_D}\Big(\Big(\mathbf{U}_D(x;\theta_2)-\mathbf{u}_D(x)\Big)+\mathbf{u}_D(x)\Big)^{r_{11}r_{13}}]^{1/r_{11}r_{13}}\\
&\times\Big[\int_{\Omega_D}\Big(|\mathbf{U}_D(x;\theta_2)-\mathbf{u}_D(x)|+|\mathbf{u}_D(x)|\Big)^{r_{11}r_{13}}d\omega_{1}(x)\Big]^{1/r_{11}r_{13}}\\
&\times\Big[\int_{\Omega_D}|\mathbf{u}_D(x)|^{r_{12}r_{14}}d\omega_{1}(x)\Big]^{1/r_{12}r_{14}}\times\Big[\int_{\Omega_D}\Big(\mathbf{U}_D(x;\theta_2)-\mathbf{u}_D(x)\Big)^{r_{12}}d\omega_{1}(x)\Big]^{1/r_{12}}.
\end{aligned}
\end{equation}
According to the inequalities (\ref{16-1}) - (\ref{16-4}), we conclude
\begin{equation}\label{16}
\begin{aligned}
&\int_{\Omega_D}\Big(|\mathbf{U}_D(x;\theta_2)\big|\mathbf{U}_D(x;\theta_2)-|\mathbf{u}_D(x)\big|\mathbf{u}_D(x)\Big)^{2}d\omega_{1}(x)\\
\leq &(\epsilon^{2}+\underset{{\Omega}_D}{sup}|  \mathbf{u}_D(x) |^{2})^2+\underset{{\Omega}_D}{sup}|  \mathbf{u}_D(x) |\epsilon^{2}+2\epsilon\underset{{\Omega}_D}{sup}|  \mathbf{u}_D(x) |(\epsilon+\underset{{\Omega}_D}{sup}|  \mathbf{u}_D(x) |)^2.
\end{aligned}
\end{equation}
For the boundary condition, we know
\begin{align}\label{17}
&\int_{\partial\Omega_D\setminus\Gamma}| \mathbf{U}_D(x;\theta_2)-\mathbf{u}_D(x)|^{2}d\omega_{2}(x)\leq C\epsilon^{2}.
\end{align}
Combining the equations (\ref{15}) - (\ref{17}), we obtain
\begin{equation}
\begin{aligned}\label{18}
J_{\Omega_D\setminus\Gamma}(\overline{\mathbf{U}})=&\Big\|f_D-\nabla\cdot \mathbf{U}_D(x;\theta_2)\Big\|_{0, \Omega_D, \omega_{1}}^{2}+\Big\|P_D(x;\theta_4)\Big\|_{0, \partial\Omega_D\setminus\Gamma, \omega_{2}}^{2}\\
&+\Big\|\frac{\mu}{\rho}K^{-1}\mathbf{U}_D(x;\theta_2)+\frac{\beta}{\rho}\big|\mathbf{U}_D(x;\theta_2)\big|\mathbf{U}_D(x;\theta_2)
+\nabla P_D(x;\theta_4)-\mathbf{g}_D\Big\|_{0, \Omega_D,\omega_{1}}^{2}\\
%\leq&\big\|\nabla\cdot\mathbf{U}_D(x;\theta_2)-\nabla\cdot\mathbf{u}_D(x)\big\|_{0,\Omega_D,\omega_{1}}^{2}+\big\|P_D(x;\theta_4)-p_D(x)\big\|_{0,\partial\Omega_D\setminus\Gamma,\omega_{2}}^{2}\\
%&+\frac{\mu}{\rho}K^{-1}\big\|\mathbf{U}_D(x;\theta_2)-\mathbf{u}_D(x)\big\|_{0,\Omega_D,\omega_{1}}^{2}+\big\|\nabla {P_D}(x;\theta_4)-\nabla p_D(x)\big\|_{0,\Omega_D,\omega_{1}}^{2}\\
%&+\frac{\beta}{\rho}\big\||\mathbf{U}_D(x;\theta_2)\big|\mathbf{U}_D(x;\theta_2)-|\mathbf{u}_D(x)\big|\mathbf{u}_D(x)\big\|_{0,\Omega_D,\omega_{1}}^{2}\\
\leq&\int_{\Omega_D}\Big|\nabla\cdot\Big(\mathbf{U}_D(x;\theta_2)-\mathbf{u}_D(x)\Big)\Big|^{2}d\omega_{1}(x)+\int_{\partial\Omega_D\setminus\Gamma}\Big(P_D(x;\theta_4)-p_D(x)\Big)^{2}d\omega_{2}(x)\\
&+\frac{\beta}{\rho}\int_{\Omega_D}\Big(\big|\mathbf{U}_D(x;\theta_2)\big|\mathbf{U}_D(x;\theta_2)-\big|\mathbf{u}_D(x)\big|\mathbf{u}_D(x)\Big)d\omega_{1}(x)\\
&+\frac{\mu}{\rho}K^{-1}\int_{\Omega_D}\Big(\mathbf{U}_D(x;\theta_2)-\mathbf{u}_D(x)\Big)^{2}d\omega_{1}(x)\\
&+\int_{\Omega_D}\Big(\nabla\mathbf{P}_D(x;\theta_4)-\nabla\mathbf{p}_D(x)\Big)^{2}d\omega_{1}(x)\\
\leq  &C\epsilon^{2}.
\end{aligned}
\end{equation}
The loss in interface is referred in  (\ref{J123}). According to the Assumption \ref{assumption1}, we know that
\begin{equation}
\begin{aligned}\label{19}
&\int_{\Gamma}\Big\vert\nabla \mathbf{U}_D(x;\theta_2)-\nabla \mathbf{u}_D(x)\Big\vert^{2}d\omega_{3}(x) \\
\leq  &\int_{\Gamma}\Big(| \mathbf{U}_D(x;\theta_2)|^{l_{7}}+| \mathbf{u}_D(x)|^{l_{8}}\Big)\Big(\mathbf{U}_D(x;\theta_2)-\mathbf{u}_D(x)\Big)^{2}d\omega_{3}(x),
\end{aligned}
\end{equation}
which can be updated by using the \emph{H$\ddot{o}$lder} inequality and \emph{Young} inequality, thus we have
\begin{equation}
\begin{aligned}\label{19-1}
&\Big[\int_{\Gamma}\Big(| \mathbf{U}_D(x;\theta_3)|^{l_{7}}+| \mathbf{u}_D|^{l_{8}}\Big)^{r_{15}}d\omega_{3}(x)\Big]^{1/r_{15}}\\
&\times\Big[\int_{\Gamma}
(\mathbf{U}_D(x;\theta_3)-\mathbf{u}_D(x))|^{2r_{16}}d\omega_{3}(x)\Big]^{1/r_{16}}\\
\leq  &\Big[\int_{\Gamma}\Big(| \mathbf{U}_D(x;\theta_3)-\mathbf{u}_D(x)|^{l_{7}}+| u_D(x)|^{l_{7}\vee l_{8}}\Big)^{r_{15}}d\omega_{3}(x)\Big]^{1/r_{15}}\\
&\times\Big[\int_{\Gamma}
\Big(\mathbf{U}_D(x;\theta_3)-\mathbf{u}_D(x)\Big)^{2r_{16}}d\omega_{3}(x)\Big]^{1/r_{16}}\\
\leq &C \Big(\epsilon^{l_7}+\underset{\Gamma}{sup}|  \mathbf{u}_D(x) |^{l_7\vee l_8})\Big)\epsilon^{2},
\end{aligned}
\end{equation}
where $\frac{1}{r_{15}}+\frac{1}{r_{16}}=1$ and $l_{7}\vee l_{8}=max\{l_{7}, l_{8}\}$.

Above all, we can obtain
\begin{equation}
\begin{split}
J_{\Gamma}(\overline{\mathbf{U}})=&\Big\|\mathbf{U}_S(x;\theta_1)\cdot\mathbf{n}_S-\mathbf{U}_D(x;\theta_2)\cdot\mathbf{n}_S\Big\|_{0, \Gamma, \omega_{3}}^{2}+\Big\|P_ S(x;\theta_3)-\nu\mathbf{n}_S\frac{\partial\mathbf{U}_S(x;\theta_1)}{\partial\mathbf{n}_S}-P_D(x;\theta_4)\Big\|_{0, \Gamma,\omega_{3}}^{2}\\
&+\Big\|-\nu\mathbf{t}\frac{\partial\mathbf{U}_S(x;\theta_1)}{\partial\mathbf{n}_S}-G\mathbf{U}_S(x;\theta_1)\cdot\mathbf{t}\Big\|_{0, \Gamma, \omega_{3}}^{2}\\
\leq&\big\|\mathbf{U}_S(x;\theta_2)-\mathbf{u}_S(x)\big\|_{0,\Gamma,\omega_{3}}^{2}\Big\|\mathbf{n}_S\Big\|_{0, \Gamma, \omega_{3}}^{2}+\big\|\mathbf{U}_D(x;\theta_2)-\mathbf{u}_D(x)\big\|_{0,\Gamma,\omega_{3}}^{2}\Big\|\mathbf{n}_S\Big\|_{0, \Gamma, \omega_{3}}^{2}\\
&+\big\|P_ S(x;\theta_3)-p_ S(x)\big\|_{0,\Gamma,\omega_{3}}^{2}+\big\|P_D(x;\theta_4)-p_D(x)\big\|_{0,\Gamma,\omega_{3}}^{2}\\
&+\nu\Big\|\mathbf{n}_S\Big\|_{0, \Gamma, \omega_{3}}^{4}\big\|\nabla\mathbf{U}_S(x;\theta_1)-\nabla\mathbf{u}_S(x;\theta_1)\big\|_{0,\Gamma,\omega_{3}}^{2}+G\Big\|\mathbf{t}\Big\|_{0, \Gamma, \omega_{3}}^{2}\big\|\mathbf{U}_S(x;\theta_1)- \mathbf{u}_S(x)\big\|_{0,\Gamma,\omega_{3}}^{2}\\
&+\nu\Big\|\mathbf{t}\Big\|_{0, \Gamma, \omega_{3}}^{2}\big\|\nabla\mathbf{U}_S(x;\theta_1)-\nabla\mathbf{u}_S(x;\theta_1)\big\|_{0,\Gamma,\omega_{3}}^{2}\Big\|\mathbf{n}_S\Big\|_{0, \Gamma, \omega_{3}}^{2}\\
\leq  &C\epsilon^{2},
\end{split}
\end{equation}
which completes the proof.
\end{proof}

\subsection{Convergence of the CDNNs to the exact solution}
In the last subsection, we have proved the convergence of the loss function. In this subsection, we remain to discuss the convergence of the CDNNs to the exact solution. According to the Galerkin method, the neural networks satisfy
\begin{align}
\nabla\cdot \mathbf{U}^n_D-f_D&=0,~~~in~\Omega_D,\label{SDFW-1}\\
\frac{\mu}{\rho}K^{-1}\mathbf{U}^n_D+\frac{\beta}{\rho}\mid\mathbf{U}^n_D\mid\mathbf{U}^n_D+\nabla P^n_D-\mathbf{g}_D&=0,~~~in~\Omega_D,\label{SDFW-2}\\
P^n_D&=0, ~~~on~\partial\Omega_D\setminus\Gamma,\label{SDFW-3}\\
-\nu\Delta\mathbf{U}^n_S+\nabla P^n_ S-\mathbf{f}_S&=0,~~~in~\Omega_S,\label{SDFW-4}\\
\nabla\cdot\mathbf{U}^n_S&=0, ~~~on~\Omega_S,\label{SDFW-5}\\
\mathbf{U}^n_S&=0, ~~~on~\partial\Omega_S\setminus\Gamma,\label{SDFW-6}\\
\mathbf{U}^n_S\cdot\mathbf{n}_S-\mathbf{U}^n_D\cdot\mathbf{n}_S&=0,~~~on~\Gamma,\label{SDFW-7}\\
P^n_ S-\nu\mathbf{n}_S\frac{\partial\mathbf{U}^n_S}{\partial\mathbf{n}_S}-P^n_D&=0,~~~on~\Gamma,\label{SDFW-8}\\
-\nu \mathbf{t}\frac{\partial\mathbf{U}^n_S}{\partial\mathbf{n}_S}-G\mathbf{U}^n_S\cdot\mathbf{t}&=0, ~~~on~\Gamma.\label{SDFW-9}
\end{align}

Based on the above system of equations, we give the following assumption and theorem to guarantee the convergency of the CDNNs to the exact solution.
\begin{assumption}\label{assumption2}
We assume $(\mathbf{u}_S,\mathbf{u}_D)\in C^{\xi}(\bar{\Omega}_S)\times C^{\xi}(\bar{\Omega}_D)$ where $\xi >2$ with itself and its first derivative bounded in $\bar{\Omega}_S\times\bar{\Omega}_D.$ Moreover, for every $n\in N$, $\mathbf{U}_S^n\times \mathbf{U}_D^n\in C^{1,2}(\bar{\Omega}_S)\cap \mathbb{X}_S\times C^{1,2}(\bar{\Omega}_D)\cap L^3(\OM_D)^2$. We assume that the subspace $\mathbb{X}_S^n\times L^3(\OM_D)^2\times L^2(\OM_S)\times\mathbb{Y}_D^n\subset [\mathfrak{C}_{\mathbf{U}_S}(\varphi)]^{d}\times[\mathfrak{C}_{\mathbf{U}_D}(\zeta)]^{d}\times\mathfrak{C}_{P_S}(\psi)\times\mathfrak{C}_{P_D}(\gamma)$ satisfies the discrete \emph{inf-sup} condition.
\end{assumption}
\begin{thm}\label{thm2}
Under the Assumption \ref{assumption1} and Theorem \ref{thm1}, the neural network ${\mathbf{U}_S^{n}}$ can converge strongly to $\mathbf{u}_S$ in $L^{2}$,  ${P_ S^{n}},~{\mathbf{U}_D^{n}}$ and ${P_D^{n}}$ can converge strongly to $p_ S, \mathbf{u}_D$ and $p_D$ in $H^{-1}$. In addition, if the sequences  $\{\mathbf{U}_S^{n}\}_{n\in \mathbb{N}}, \{P_ S^{n}\}_{n\in \mathbb{N}}, \{\mathbf{U}_D^{n}\}_{n\in \mathbb{N}}$ and $ \{P_D^{n}\}_{n\in \mathbb{N}}$ are uniformly bounded and equicontinuous in $\Omega_S$ and $\Omega_D$, they can converge to $\mathbf{u}_S,~ p_ S,\mathbf{ u}_D$ and $p_D$ respectively.
\end{thm}
\begin{proof}

Firstly, we give the weak formulation for (\ref{SDFW-1})-(\ref{SDFW-9}). Multiplying (\ref{SDFW-4}) by $\mathbf{V}^n_ S \in  \mathbb{X}_S\bigcap[\mathfrak{C}_{\mathbf{U}_S}(\varphi)]^{d}$ and (\ref{SDFW-5}) by $Q^n_ S\in \mathfrak{C}_{P_ S}(\psi)$ and integration by parts yields
\begin{align}
\nu(\nabla \mathbf{U}^n_ S, \nabla \mathbf{V}^n_S)_{\Omega_S} -\nu <\nabla \mathbf{U}^n_ S\cdot \mathbf{n}_S, \mathbf{V}^n_ S>_{\GM}-(P^n_ S, \nabla \cdot \mathbf{V}^n_ S)_{\Omega_S}+<P^n_ S, \mathbf{V}^n_ S\cdot \mathbf{n}_S>_{\GM}&=(\mathbf{f}_S,\mathbf{V}^n_ S)_{\Omega_S},\label{20}\\
(\nabla \cdot\mathbf{U}^n_S, Q^n_ S)_{\OM_S}&=0.\label{21}
\end{align}

Multiplying (\ref{SDFW-1}) by $Q^n_ D \in  \mathbb{Y}_D\bigcap\mathfrak{C}_{P_D}(\gamma)$ and (\ref{SDFW-2}) by $\mathbf{V}^n_ D \in  [\mathfrak{C}_{\mathbf{U}_D}(\zeta)]^{d}$, it then follows from integration by parts that
\begin{align}
-(\mathbf{U}^n_ D, \nabla \mathbf{Q}^n_D)_{\Omega_D} + <\mathbf{U}^n_D\cdot \mathbf{n}_D, \mathbf{Q}^n_D>_{\GM}&=(f_D,\mathbf{Q}^n_D)_{\Omega_D},\label{22}\\
\frac{\mu}{\rho}(K^{-1}\mathbf{U}^n_D,\mathbf{V}^n_D)_{\Omega_D}+\frac{\beta}{\rho}(|\mathbf{U}^n_D|\mathbf{U}^n_D,\mathbf{V}^n_D)_{\Omega_D}-(P^n_D,\nabla\cdot\mathbf{V}^n_D)_{\OM_D}+<P^n_D, \mathbf{V}^n_D\cdot \mathbf{n}_D>_{\GM}&=(\mathbf{g}_D,\mathbf{V}^n_D)_{\OM_D}.\label{23}
\end{align}
Considering the interface conditions, simple algebraic calculation yields
$$<\nabla \mathbf{U}^n_ S\mathbf{n}_S, \mathbf{V}^n_S>_{\GM}=<\mathbf{n}_S\nabla \mathbf{U}^n_ S\cdot\mathbf{n}_S, \mathbf{V}^n_S\cdot\mathbf{n}_S>_{\GM}+<\mathbf{n}_S\nabla \mathbf{U}^n_ S\cdot\mathbf{t}, \mathbf{V}^n_S\cdot\mathbf{t}>_{\GM}$$
which gives by employing interface conditions (\ref{SDFW-8}) and (\ref{SDFW-9})
\begin{align}
<(P^n_ SI-\nu\nabla\mathbf{U}^n_ S)\mathbf{n}_S,\mathbf{V}^n_ S>_{\GM}=<P^n_D,\mathbf{V}^n_S\cdot\mathbf{n}_S>_{\GM}+G<\mathbf{U}^n_S\cdot\mathbf{t},\mathbf{V}^n_S\cdot\mathbf{t}>_{\GM}.\label{4647}
\end{align}

For convenience of presentation, we introduce the nonlinear operator $A : L^3(\OM_ D)^2 \rightarrow L^{3/2}(\OM_ D)^2$ defined by
\begin{align}\label{24}
A(\mathbf{V})=\frac{\mu K^{-1}}{\rho}\mathbf{V}+\frac{\beta}{\rho}|\mathbf{V}|\mathbf{V}.
\end{align}

The definition of (\ref{24}) gives
$$(A(\mathbf{U}^n_D),\mathbf{U}^n_D)_{\OM_D}\geq C(\|\mathbf{U}^n_D\|^2_{0}+\|\mathbf{U}^n_D\|^3_{L^3}).$$

%Taking $v_S = u_S$ and $Q_ S=P_ S$, By using the definition of the $H^1$ norm in (\ref{20}) and (\ref{21}), we can obtain
%\begin{align}\label{25}
%\|U^n\|_1\leq C \|f_S\|_{0,\OM_S}.
%\end{align}

According to the Assumption \ref{assumption2} and (\ref{23}),
\begin{align*}
\|P^n_D\|_{0}&\leq C\sup_{\textbf{V}^n_D\in L^3(\OM_D)^2\bigcap[\mathfrak{C}_{\textbf{U}_D}(\zeta)]^{d}}\frac{(A(\textbf{U}^n_D),\textbf{V}^n_D)_{\OM_D}-(g_D,\textbf{V}^n_D)_{\OM_D}}{\|\textbf{V}^n_D\|_{L^3(\OM_D)}}\\
&\leq C(\|\mathbf{U}^n_D\|_{0}+\|\textbf{U}^n_D\|^2_{L^3}+\|\mathbf{g}_D\|_{0}).
\end{align*}

Taking $\mathbf{V}^n_S = \mathbf{U}^n_S, Q^n_ S=P^n_ S, \mathbf{V}^n_ D = \mathbf{U}^n_ D$ and $Q^n_ D = P^n_D$ in (\ref{20})-(\ref{23}) and adding the resulting equations (\ref{4647}) yields
\begin{align*}
&\nu(\nabla \mathbf{U}^n_ S, \nabla \mathbf{U}^n_S)_{\Omega_S}+G<\mathbf{U}^n_S\cdot\mathbf{t},\mathbf{U}^n_S\cdot\mathbf{t}>_{\GM}
+\frac{1}{2}(A(\mathbf{U}^n_D),\mathbf{U}^n_D)_{\OM_D}\\
&= (\mathbf{f}_S,\mathbf{U}^n_ S)_{\Omega_S}+(f_D,P^n_D)_{\Omega_D}+\frac{1}{2}(\mathbf{g}_D,\mathbf{U}^n_D)_{\OM_D}.
\end{align*}

According to the definition of the $H^1$ norm, the \emph{Young} inequality and the \emph{$Poincar\acute{e}$} inequality, we can obtain
\begin{align*}
&\|\mathbf{U}^n_ S\|_{1}^2+\|\mathbf{U}^n_D\|^2_{0}+\|\mathbf{U}^n_D\|^3_{L^3}\\
\leq& C(\|\mathbf{f}_S\|_{0}\|\mathbf{U}^n_ S\|_{0}+\|f_D\|_{0}\|P^n_D\|_{0}+\|\mathbf{g}_D\|_{0}\|\mathbf{U}^n_D\|_{0})\\
\leq& C(\|\mathbf{f}_S\|_{0}\|\mathbf{U}^n_ S\|_{0}+\|f_D\|_{0}(\|\mathbf{U}^n_D\|_{0}+\|\mathbf{U}^n_D\|^2_{L^3}+\|\mathbf{g}_D\|_{0})+\|\mathbf{g}_D\|_{0}\|\mathbf{U}^n_D\|_{0})\\
\leq& C(\|f_D\|^{2}_{0}+\|f_D\|^{3}_{0}+\|\mathbf{f}_S\|^{2}_{0}+\|\mathbf{g}_D\|^{2}_{0}).
\end{align*}

%\textcolor[rgb]{1.00,0.00,0.00}{Lemma 3.6}(discrete trace inequality) The following estimate holds:
%$$\|q_h\|_{0(\GM)}\leq C\|q_h\|_{0(Z_D)}~~~\forall q_h\in U_D.$$
%
%\textcolor[rgb]{1.00,0.00,0.00}{Lemma 3.7}(generalized Pointcar$\acute{e}$-Friedrichs inequality) The following estimate holds:
%$$\|q_h\|_{0(\OM_D)}\leq C\|q_h\|_{0(Z_D)}~~~\forall q_h\in U_D.$$
%
%\textcolor[rgb]{1.00,0.00,0.00}{Corollary 3.1}. The following estimate holds:
%$$\|v_h\|_{0(\GM)}\leq C\|v_h\|_{0(h)}~~~\forall v_h\in [U_S]^2.$$

Thus we have
\begin{align*}
\|P^n_D\|_{0}&\leq C(\|\mathbf{U}^n_D\|_{0}+\|\mathbf{U}_D\|^2_{L^3}+\|\mathbf{g}_D\|_{0})\\
&\leq C(\|f_D\|_{0}+\|f_D\|^{2}_{0}+\|\mathbf{f}_S\|_{0}+\|\mathbf{f}_S\|^{2}_{0}+\|\mathbf{g}_D\|_{0}+\|\mathbf{g}_D\|^{2}_{0}).
\end{align*}

According to the Assumption \ref{assumption2}, (\ref{20}) - (\ref{21}) and  (\ref{4647})
\begin{align*}
\|P^n_ S\|_{0}&\leq C\sup_{\mathbf{V}^n_S\in [\mathbb{X}_S]^2\bigcap[\mathfrak{C}_{\mathbf{U}_S}(\varphi)]^{d}}\frac{(\mathbf{f}_S,\mathbf{V}^n_S)-(\nabla \mathbf{U}^n_S, \nabla\mathbf{V}^n_S)-<P^n_D,\mathbf{V}^n_S\cdot \mathbf{n}_S>_{\GM}-G<\mathbf{U}^n_S\cdot \mathbf{t},\mathbf{V}^n_S\cdot \mathbf{t}>_{\GM}}{\|\mathbf{V}^n_S\|_{0}}\\
&\leq C(\|\textbf{f}_S\|_{0}+\|\nabla \mathbf{U}^n_S\|_{0}+\|P^n_D\|_{0}+\|\mathbf{U}^n_ S\|_{0})\\
&\leq C(\|f_D\|_{0}+\|f_D\|^{2}_{0}+\|\mathbf{f}_S\|_{0}+\|\mathbf{f}_S\|^{2}_{0}+\|\mathbf{g}_D\|_{0}+\|\mathbf{g}_D\|^{2}_{0}).
\end{align*}

Up to now, we obtain
 \begin{equation}\label{u1}
\{\mathbf{U}_S^{n}\}_{n\in N}~~is~~uniformly~~bounded~~in~~ H^1(\Omega_S).
\end{equation}

 By using the uniformly boundedness of $\mathbf{U}_S^{n}$, we can extract a subsequence $\{\mathbf{U}_S^{n}\}_{n\in\mathbb{N}}$ of $\mathbf{U}_S^{n}$ which converge weakly in $H^{1}(\Omega_S)$. Due to the compact embedding $H^{1}(\Omega_S)\hookrightarrow L^{2}(\Omega_S)$,
we have $\underset{n\rightarrow\infty}{\lim}\parallel \mathbf{U}_S^{n}-\mathbf{u}_S\parallel_{0,\Omega_S}=0$.

Similarly, we know $L^3(\OM_D)\subset L^2(\OM_D)$, which implies that
  \begin{equation}\label{u1}
\{\mathbf{U}_D^{n}\}_{n\in N}~~is~~uniformly~~bounded~~in~~ L^2(\Omega_D).
\end{equation}
Due to the compact embedding $L^{2}(\Omega_D)\hookrightarrow H^{-1}(\Omega_D)$, we have $\underset{n\rightarrow\infty}{\lim}\parallel \mathbf{U}_D^{n}-\mathbf{u}_D\parallel_{-1,\Omega_D}=0$.

The convergence of the $P_ S^n$ and $P_D^n$ as follows
$$\underset{n\rightarrow\infty}{\lim}\parallel P_ S^{n}-P_ S\parallel_{-1,\Omega_S}=0,~~~\underset{n\rightarrow\infty}{\lim}\parallel P_D^{n}-p_D\parallel_{-1,\Omega_D}=0.$$

For all these reasons, $\{\mathbf{U}_S^{n}\}_{n\in \mathbb{N}}$ converge strongly to $\mathbf{u}_S$ in $L^{2}$,  $\{P_ S^{n}\}_{n\in \mathbb{N}}, \{U_D^{n}\}_{n\in \mathbb{N}}$ and $ \{P_D^{n}\}_{n\in \mathbb{N}}$ converge strongly to $p_ S, \mathbf{u}_D$ and $p_D$ in $H^{-1}$. More generally, by the well known Arzel$\grave{a}$-Ascoli theorem we can conclude that $\{\mathbf{U}_S^{n}\}_{n\in \mathbb{N}}, \{P_ S^{n}\}_{n\in \mathbb{N}},\{\mathbf{U}_D^{n}\}_{n\in \mathbb{N}}$ and $\{P_D^{n}\}_{n\in \mathbb{N}}$ converge uniformly to $\mathbf{u}_S, p_ S, \mathbf{u}_D$ and $p_D$ respectively.
\end{proof}
\section{Numerical Experiments}
The section presents several numerical tests to confirm the proposed theoretical results. We start with three examples with known exact solution to test the efficiency of the proposed method, where the permeability for the third example is highly oscillatory. Then, the fourth example with no exact solution shows the application of the proposed method to high contrast permeability problem. This section concludes with a physical flow. The numerical examples presented below could violate the interface conditions (\ref{SDF-8}) and (\ref{SDF-9}) \cite{Zhao}, that is, (\ref{SDF-8}) and (\ref{SDF-9}) are replaced by
\begin{align}
p_S-\nu\mathbf{n}_S\frac{\partial\mathbf{u}_S}{\partial\mathbf{n}_S}=p_D+g_1,~~~on~\Gamma, \label{newcon1}\\
-\nu \mathbf{t}\frac{\partial\mathbf{u}_S}{\partial\mathbf{n}_S}=G\mathbf{u}_S\cdot\mathbf{t}+g_2, ~~~on~\Gamma, \label{newcon2}
\end{align}
to deal with this case, the variational formulation has only a small change: The equation \ref{4647}  now includes the two terms  $- <g_1, \mathbf{V}_S \cdot  \mathbf{n}_S>_ \Gamma   -<g_2, \mathbf{V}_S \cdot  \mathbf{t}>_\Gamma$  on the right side. In addition, we utilize 16 neurons in each hidden layer and apply the relative $L^1$ error $(errL^1:  \frac{\|\mathbf{r}-\mathbf{R}\|_{L^1}}{\|\mathbf{r}\|_{L^1}})$ and relative $L^2$ error $(errL^2:  \frac{\|\mathbf{r}-\mathbf{R}\|_{0}}{\|\mathbf{r}\|_{0}})$  to reflect the accuracy between the results of the CDNNs and the exact solution ($\mathbf{r}$: the exact solution; $\mathbf{R}$: the neural network).

\subsection{Test 1}
In this subsection we study the performance of the CDNNs for the benchmark problem presented in \cite{Zhao}. This problem is defined for $\Omega_S = (0, 1)^2,~\Omega_D = (0, 1)\times (1, 2)$ and the interface $\Gamma  = \{ 0 < x < 1,~ y = 1\}$ as

$$\mathbf{u}_S = \left(
                   \begin{array}{c}
                    x^2\pi sin(2\pi y)(x-1)^2 \\
                    -2x sin(y\pi)^2(2x-1)(x-1) \\
                   \end{array}
                 \right)
,~~p_S =(cos(1)-1)sin(1)+cos(y)sin(x)$$
and
$$\mathbf{u}_D = \left(
                   \begin{array}{c}
                     sin(\pi x)sin(\pi y)\\
                    -2xsin(y\pi)^2(2x-1)\\
                   \end{array}
                 \right)
,~~p_D =sin(\pi x)cos(\pi y). $$
Similar to \cite{Zhao}, we fix $\mathbf{K}$ to be the identity tensor in $\mathbb{R}^{ 2\times 2},~\mu  = \rho  = \beta  = \nu  = 1$. Due to the interface conditions (\ref{SDF-8}) and (\ref{SDF-9}) are violated, we exploit the interface conditions (\ref{newcon1}) and (\ref{newcon2}), where $g_1$ and $g_2$ can be computed by the exact solution. Specifically, the errors converge as the hidden layer increases in Figure \ref{figure3a}. Figure \ref{figure3b} reveals that the change of data has no significant influence on errors once the size is larger than $10^2$. In particular, Figures \ref{figure4} - \ref{figure5} and Table \ref{tabel2} show the details of the results, this is consistent with our theory.

\begin{figure}
\centering
\subfigure[400 training points]{
\begin{minipage}[t]{0.9\linewidth}\label{figure3a}
\centering
\includegraphics[width=3.5in]{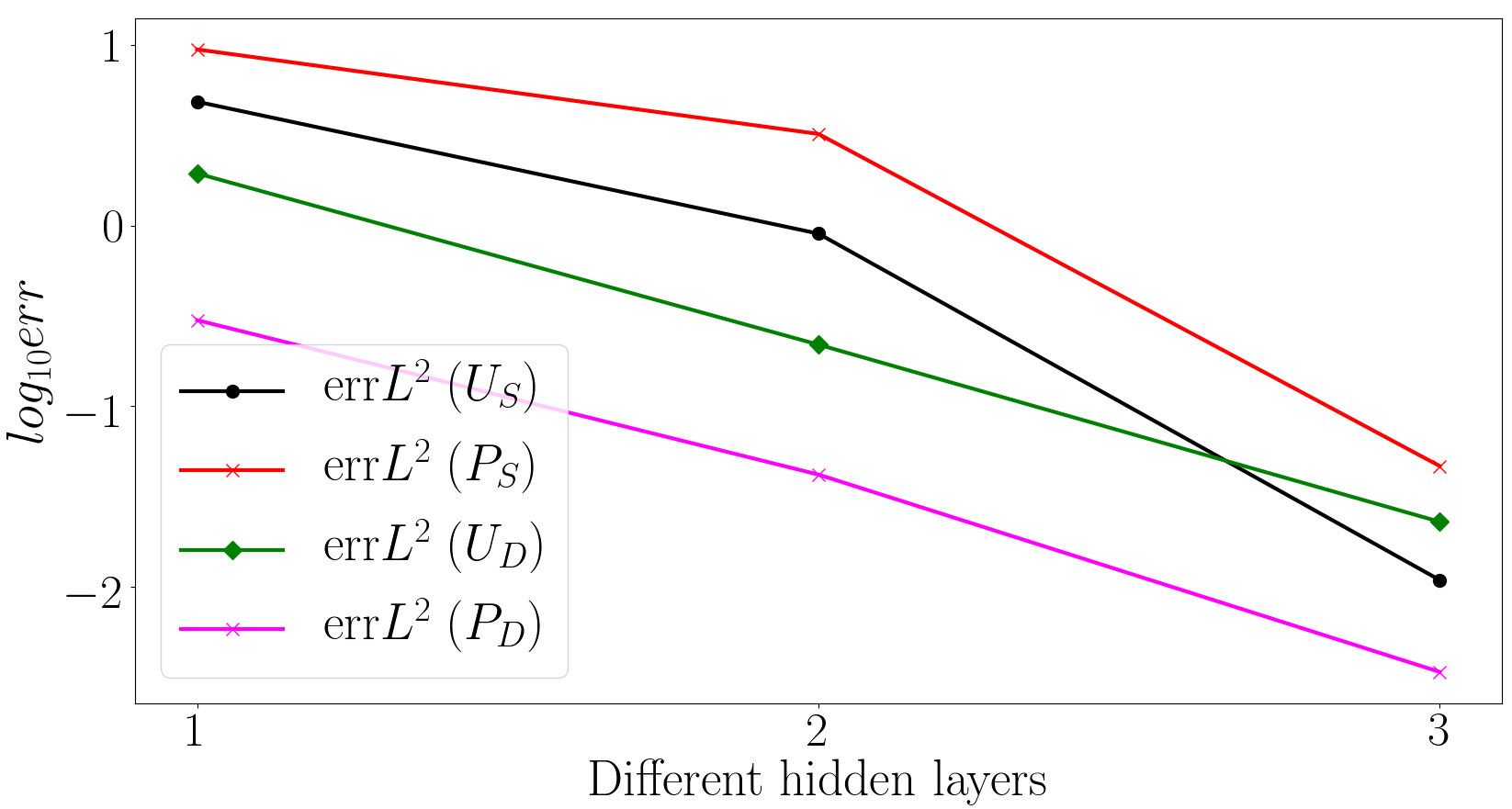}
\end{minipage}}\\
\subfigure[one hidden layer]{
\begin{minipage}[t]{0.9\linewidth}\label{figure3b}
\centering
\includegraphics[width=3.5in]{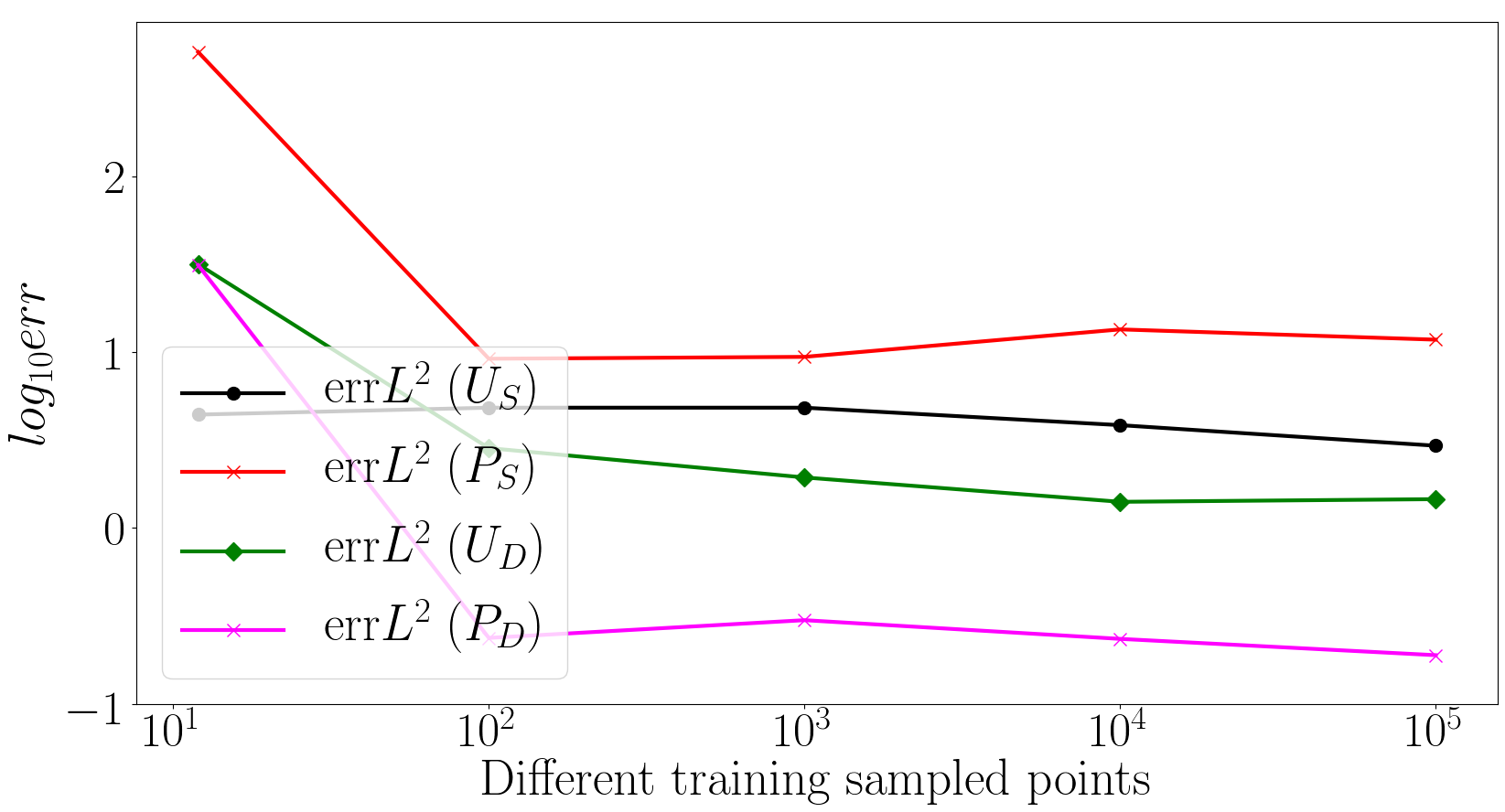}
\end{minipage}}
\caption{The influence of different hidden layers and different training data on err$L^2$ (Test 1).}

\end{figure}

\begin{table}
\centering
\caption{The relative errors of Test 1.}
\begin{tabular}{c@{\extracolsep{0.2em}}c@{\extracolsep{0.2em}}c@{\extracolsep{0.2em}}c@{\extracolsep{0.2em}}c@{\extracolsep{0.2em}}c}
\hline
   & &400  sampled points          \\ \hline
1 layer    &$\mathbf{U}_S$          &$P_S$             &$\mathbf{U}_D$             &$P_D$   \\ \hline
$errL^1$    \ \ \  &$2.49\times10^{0}$   \ \ \ &$9.15\times10^{0}$   \ \ \    &$8.72\times10^{0}$    \ \ \    &$2.74\times10^{-1}$  \\
$errL^2$    \ \ \  &$4.84\times10^{0}$  \ \ \ &$9.42\times10^{0}$    \ \ \  &$1.94\times10^{0}$      \ \ \    &$2.99\times10^{-1}$ \\\hline\hline
2 layers    &$\mathbf{U}_S$          &$P_S$             &$\mathbf{U}_D$             &$P_D$      \\ \hline
$errL^1$   \ \ \  &$4.85\times10^{-1}$   \ \ \ &$3.59\times10^{0}$   \ \ \    &$9.62\times10^{-2}$     \ \ \    &$4.03\times10^{-2}$ \\
$errL^2$    \ \ \  &$9.01\times10^{-1}$  \ \ \ &$3.21\times10^{0}$    \ \ \  &$2.19\times10^{-1}$      \ \ \    &$4.18\times10^{-2}$ \\\hline\hline
3 layers   &$\mathbf{U}_S$          &$P_S$             &$\mathbf{U}_D$             &$P_D$    \\ \hline
$errL^1$    \ \ \  &$5.80\times10^{-3}$   \ \ \ &$5.26\times10^{-2}$   \ \ \    &$1.01\times10^{-2}$     \ \ \    &$3.25\times10^{-3}$ \\
$errL^2$    \ \ \  &$1.09\times10^{-2}$  \ \ \ &$4.66\times10^{-2}$    \ \ \  &$2.29\times10^{-2}$     \ \ \    &$3.38\times10^{-3}$  \\\hline\hline\
\end{tabular}
\label{tabel2}
\end{table}

\begin{figure}
\centering
\includegraphics[width=6in]{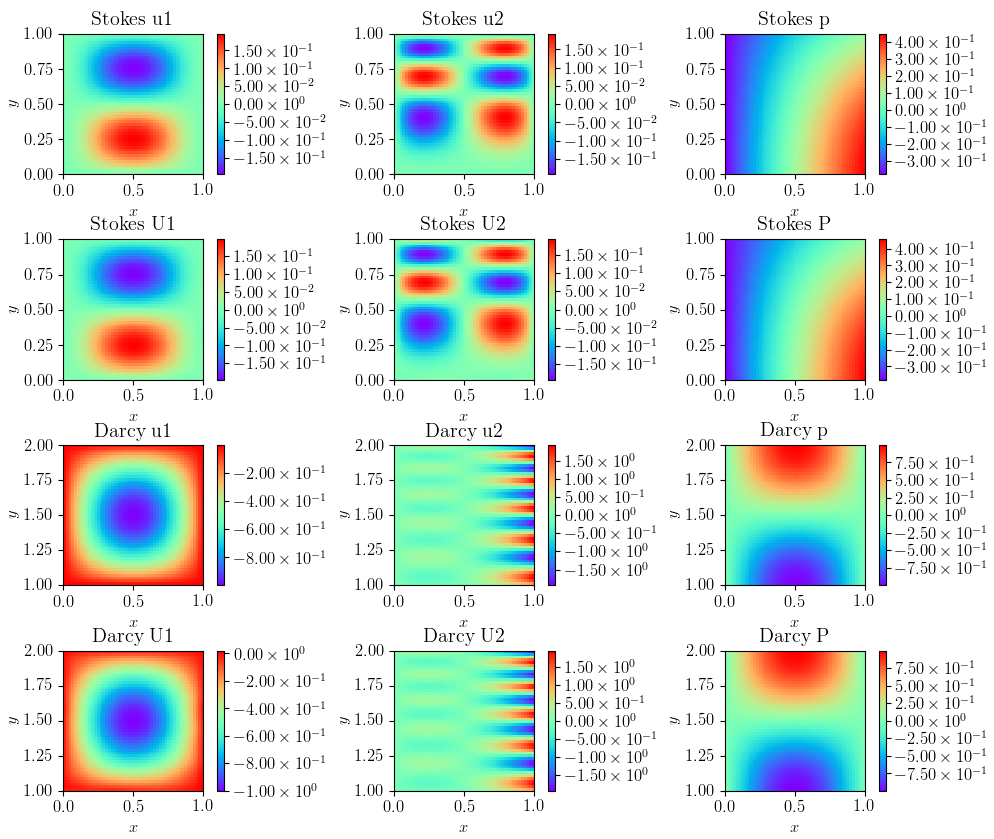}
\caption{The contrast of the exact solution and the CDNNs (Test 1).}
\label{figure4}
\end{figure}

\begin{figure}[H]
\centering
\includegraphics[width=6in]{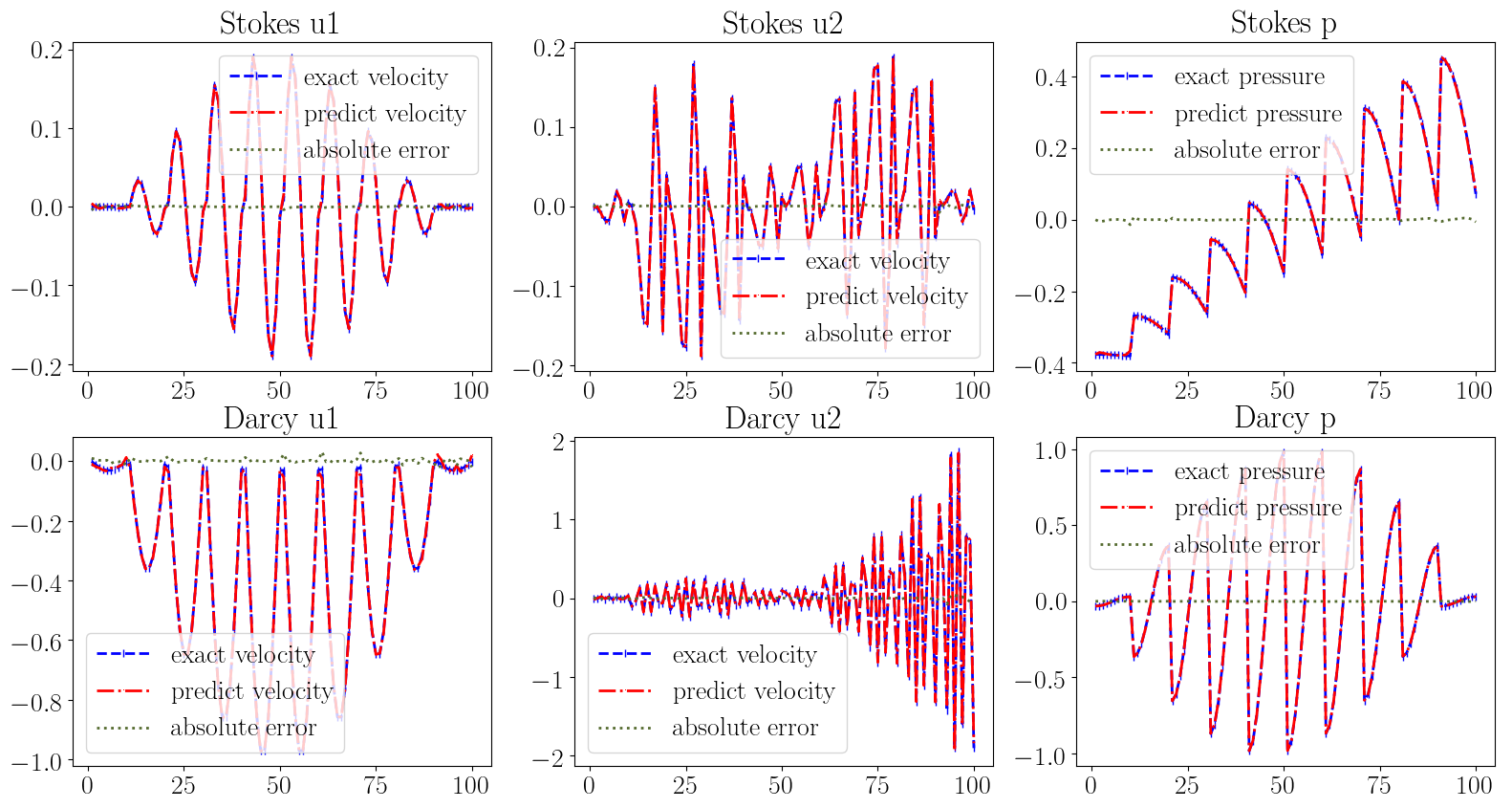}
\caption{ The point-wise errors (Test 1).}
\label{figure5}
\end{figure}

\subsection{Test 2}
In this example, we consider $\Omega_S = (0, 1)^2,~\Omega_D = (0, 1)\times (1, 2)$ and the interface $\Gamma  = \{ 0 < x < 1,~ y = 1\}$ with an analytical solution presented in \cite{Zhao}. We set $\mathbf{K}$ to be the identity tensor in $\mathbb{R}^{ 2\times 2},~\mu  = \rho  = \beta  = \nu  = 1$, and the exact solution is given by

$$\mathbf{u}_S = \left(
                   \begin{array}{c}
                     -cos^2(\frac{\pi y}{2})sin(\frac{\pi x}{2}) \\
                     \frac{1}{4}cos(\frac{\pi x}{2})(sin(\pi y)+\pi y) \\
                   \end{array}
                 \right)
,~~p_S =\frac{\pi }{4} cos(\frac{\pi x}{2})(y-2cos(\frac{\pi y}{2})^2) $$
and
$$\mathbf{u}_D = \left(
                   \begin{array}{c}
                     -\frac{1}{8}sin(\frac{\pi x}{2}) \\
                     \frac{1}{4}\pi cos(\frac{\pi x}{2}) \\
                   \end{array}
                 \right)
,~~p_D =-\frac{\pi }{4} cos(\frac{\pi x}{2})y. $$
\begin{table}[ht]
\centering
\caption{The relative errors of Test 2.}
\begin{tabular}{c@{\extracolsep{0.2em}}c@{\extracolsep{0.2em}}c@{\extracolsep{0.2em}}c@{\extracolsep{0.2em}}c@{\extracolsep{0.2em}}c}
\hline
   & &400 sampled points           \\ \hline
1 layer   &$\mathbf{U}_S$          &$P_S$             &$\mathbf{U}_D$             &$P_D$     \\ \hline
$errL^1$    \ \ \  &$1.82\times10^{-2}$   \ \ \ &$1.19\times10^{-1}$   \ \ \    &$1.57\times10^{-2}$    \ \ \    &$1.08\times10^{-2}$  \\
$errL^2$    \ \ \  &$3.66\times10^{-2}$  \ \ \ &$1.23\times10^{-1}$    \ \ \  &$4.62\times10^{-2}$      \ \ \    &$1.11\times10^{-2}$ \\\hline\hline
2 layers   &$\mathbf{U}_S$          &$P_S$             &$\mathbf{U}_D$             &$P_D$     \\ \hline
$errL^1$    \ \ \  &$2.27\times10^{-4}$   \ \ \ &$1.74\times10^{-3}$   \ \ \    &$1.04\times10^{-4}$     \ \ \    &$4.67\times10^{-5}$ \\
$errL^2$    \ \ \  &$4.21\times10^{-4}$  \ \ \ &$2.00\times10^{-3}$    \ \ \  &$3.18\times10^{-4}$      \ \ \    &$5.55\times10^{-5}$ \\\hline\hline
3 layers   &$\mathbf{U}_S$          &$P_S$             &$\mathbf{U}_D$             &$P_D$     \\ \hline
$errL^1$    \ \ \  &$1.65\times10^{-4}$   \ \ \ &$1.01\times10^{-3}$   \ \ \    &$1.13\times10^{-4}$     \ \ \    &$7.69\times10^{-5}$ \\
$errL^2$    \ \ \  &$3.37\times10^{-4}$  \ \ \ &$1.50\times10^{-3}$    \ \ \  &$3.44\times10^{-4}$     \ \ \    &$8.32\times10^{-5}$  \\\hline\hline\
\end{tabular}
\label{tabel1}
\end{table}

Naturally, the corresponding $\mathbf{f}_S,~f_D$, and $\mathbf{g}_D$ can be calculated by the exact solution. Note that this example satisfies the interface conditions (\ref{SDF-7}) - (\ref{SDF-9}). According to Test 1, we choose appropriate data and hidden layer to solve the second example. Figures \ref{figure1} - \ref{figure2} and Table \ref{tabel1} show the accuracy of the CDNNs for solving the coupled problems in detail.

\begin{figure}
\centering
\includegraphics[width=6in]{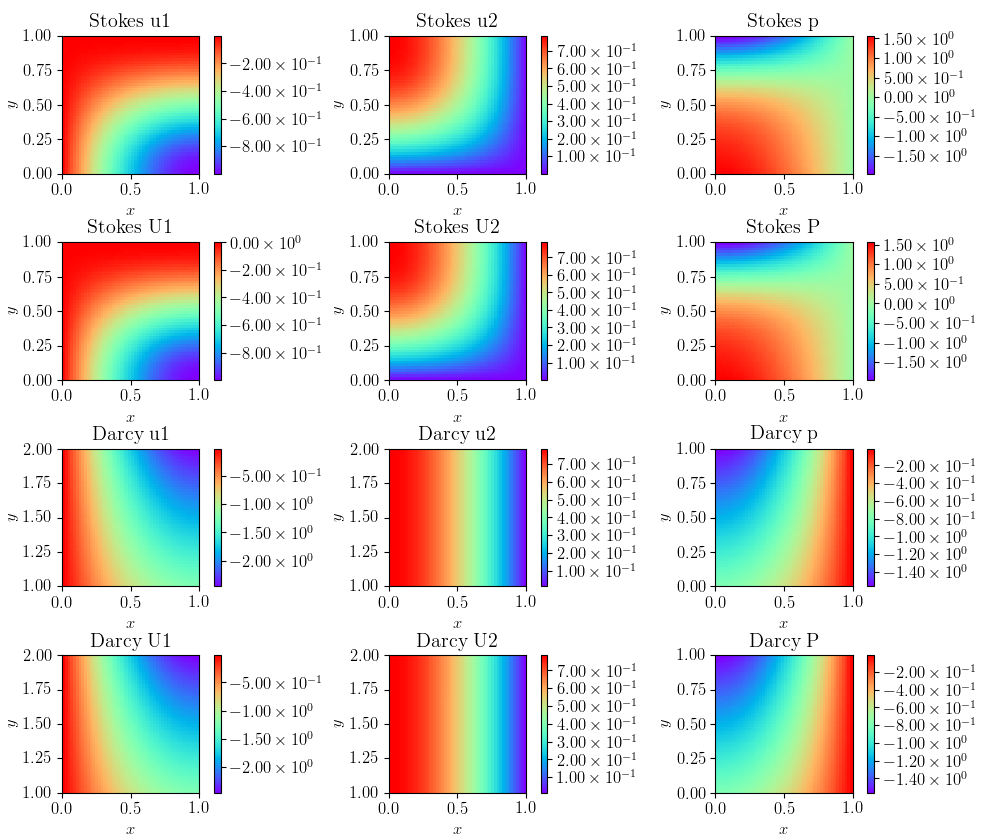}
\caption{The contrast of the exact solution and the results of the CDNNs (Test 2).}
\label{figure1}
\end{figure}

\begin{figure}[H]
\centering
\includegraphics[width=6in]{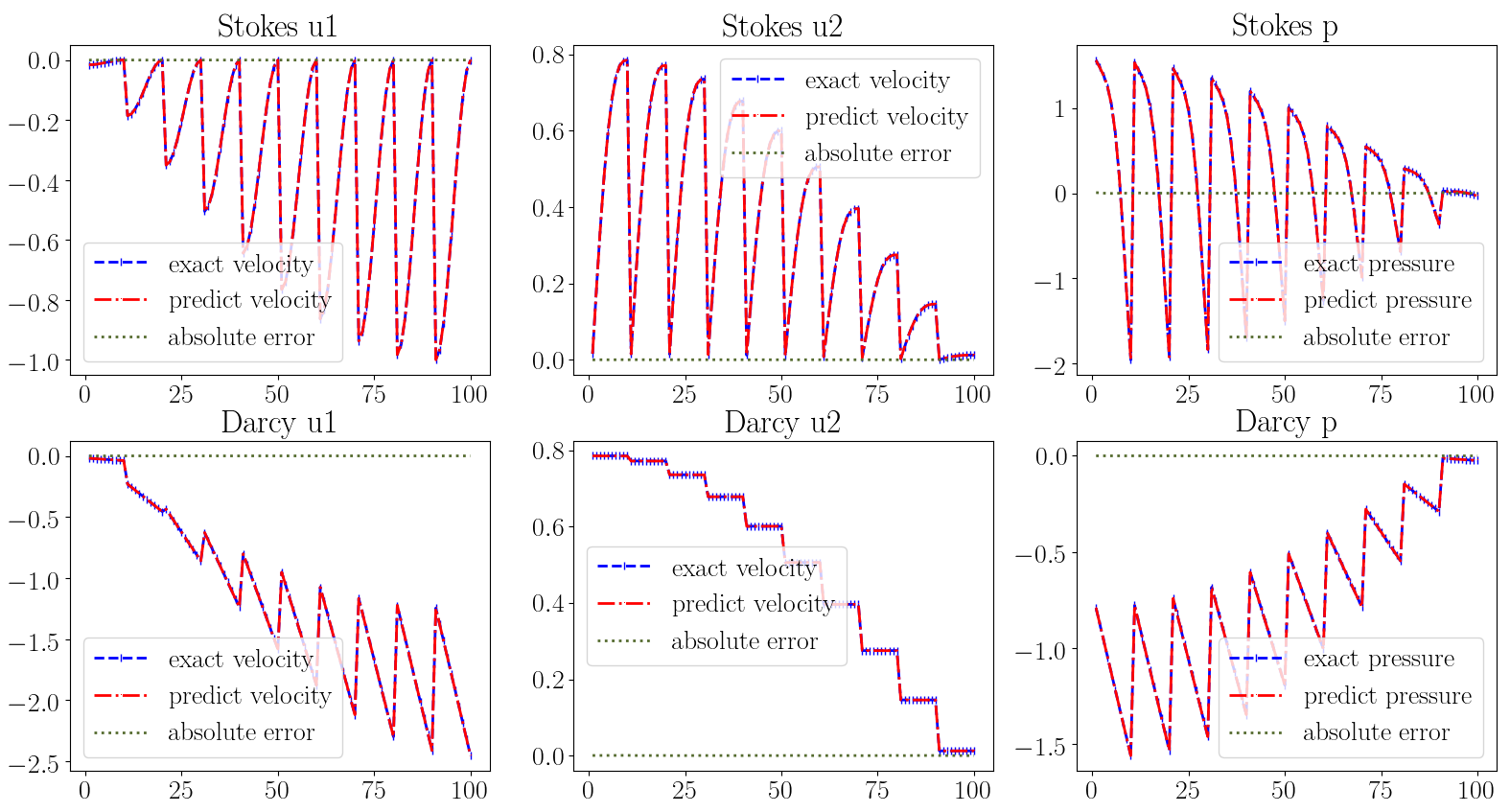}
\caption{ The point-wise errors (Test 2).}
\label{figure2}
\end{figure}

\subsection{Test 3}
In this subsection, we solve coupling of the Stokes and Darcy-Forchheimer problems with highly oscillatory permeability over domains $\Omega_S = (0, 1)\times (0, 1/2),~\Omega_D = (0, 1)\times (1/2, 1)$ and the interface $\Gamma  = \{ 0 < x < 1,~ y = 1/2\}$ presented in \cite{Zhao}. Here we set $\mu  = \rho  = \beta  = \nu  = 1$, $K^{-1} = \varrho I$ and $\varrho$  is defined by
$$ \varrho=\frac{2+1.8sin(2\pi x/\varepsilon)}{2+1.8sin(2\pi y/\varepsilon)}+\frac{2+1.8sin(2\pi y/\varepsilon)}{2+1.8sin(2\pi x/\varepsilon)},$$
where $\varepsilon  = 1/16$. The profile of $\varrho$  is shown in Figure \ref{figurehigh}. The exact solution is given by
$$\mathbf{u}_S = \left(
                   \begin{array}{c}
                    16ycos(\pi x)^2(y^2-0.25) \\
                    8\pi cos(\pi x)sin(\pi x)(y^2-0.25)^2 \\
                   \end{array}
                 \right)
,~~p_S =x^2$$
and
$$\mathbf{u}_D = \left(
                   \begin{array}{c}
                     sin(2\pi x)cos(2\pi y)\\
                    -cos(2\pi x)sin(2\pi y)\\
                   \end{array}
                 \right)
,~~p_D =cos(2\pi x)cos(2\pi y).$$

We calculate the relative errors in Table \ref{tabel3} to reflect ability of the CDNNs for solving the coupled problems with highly oscillatory permeability. Figures \ref{figure6} - \ref{figure7} reveal that the CDNNs handle the highly oscillatory permeability coupled problems without losing accuracy.

\begin{figure}[H]
\centering
\includegraphics[width=2in]{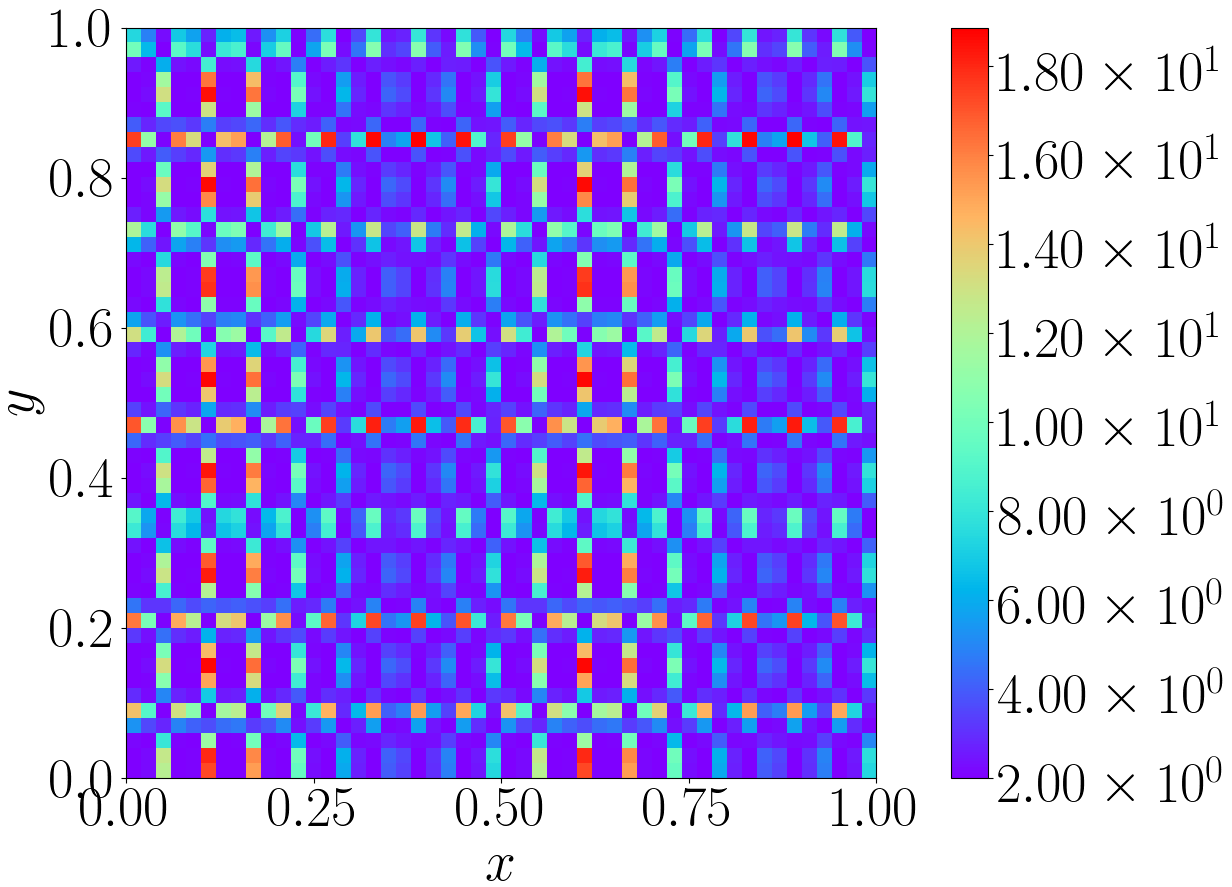}
\caption{\emph{The value of $\varrho$ (Test 3).}}
\label{figurehigh}
\end{figure}
\begin{table}[H]
\centering
\caption{The relative errors of Test 3.}
\begin{tabular}{c@{\extracolsep{0.2em}}c@{\extracolsep{0.2em}}c@{\extracolsep{0.2em}}c@{\extracolsep{0.2em}}c@{\extracolsep{0.2em}}c}
\hline
   & &400 sampled points            \\ \hline
1 layer   &$\mathbf{U}_S$          &$P_S$             &$\mathbf{U}_D$             &$P_D$     \\ \hline
$errL^1$   \ \ \  &$3.59\times10^{-1}$   \ \ \ &$5.03\times10^{0}$   \ \ \    &$5.52\times10^{-2}$    \ \ \    &$7.59\times10^{-2}$  \\
$errL^2$    \ \ \  &$6.79\times10^{-1}$  \ \ \ &$5.20\times10^{0}$    \ \ \  &$1.73\times10^{-1}$      \ \ \    &$7.07\times10^{-2}$ \\\hline\hline
2 layers   &$\mathbf{U}_S$          &$P_S$             &$\mathbf{U}_D$             &$P_D$     \\ \hline
$errL^1$   \ \ \  &$8.42\times10^{-4}$   \ \ \ &$9.41\times10^{-3}$   \ \ \    &$1.15\times10^{-3}$     \ \ \    &$1.32\times10^{-3}$ \\
$errL^2$    \ \ \  &$1.65\times10^{-3}$  \ \ \ &$1.05\times10^{-2}$    \ \ \  &$3.43\times10^{-3}$      \ \ \    &$1.56\times10^{-3}$ \\\hline\hline
3 layers   &$\mathbf{U}_S$          &$P_S$             &$\mathbf{U}_D$             &$P_D$     \\ \hline
$errL^1$    \ \ \  &$1.89\times10^{-4}$   \ \ \ &$3.04\times10^{-3}$   \ \ \    &$2.97\times10^{-4}$     \ \ \    &$6.70\times10^{-5}$ \\
$errL^2$    \ \ \  &$3.65\times10^{-4}$  \ \ \ &$3.37\times10^{-3}$    \ \ \  &$8.98\times10^{-4}$     \ \ \    &$8.40\times10^{-5}$  \\\hline\hline\
\end{tabular}
\label{tabel3}
\end{table}
\begin{figure}
\centering
\includegraphics[width=5.5in]{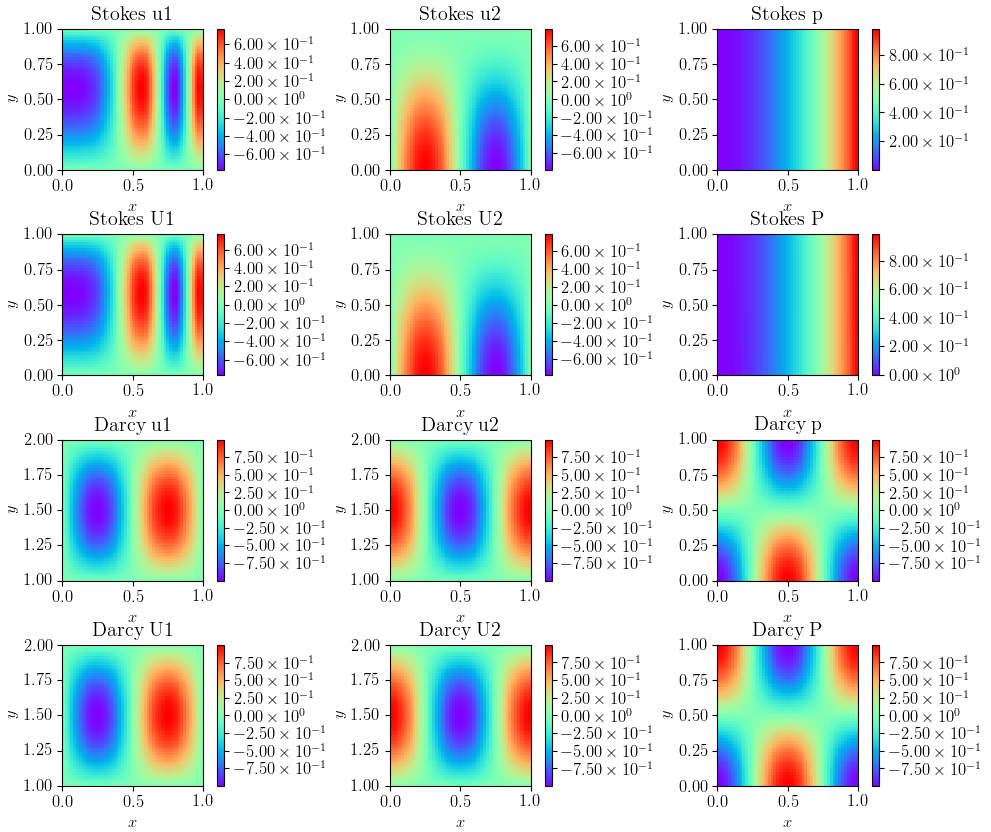}
\caption{The contrast of the exact solution and  the results of the CDNNs (Test 3).}
\label{figure6}
\end{figure}
\begin{figure}[H]
\centering
\includegraphics[width=6in]{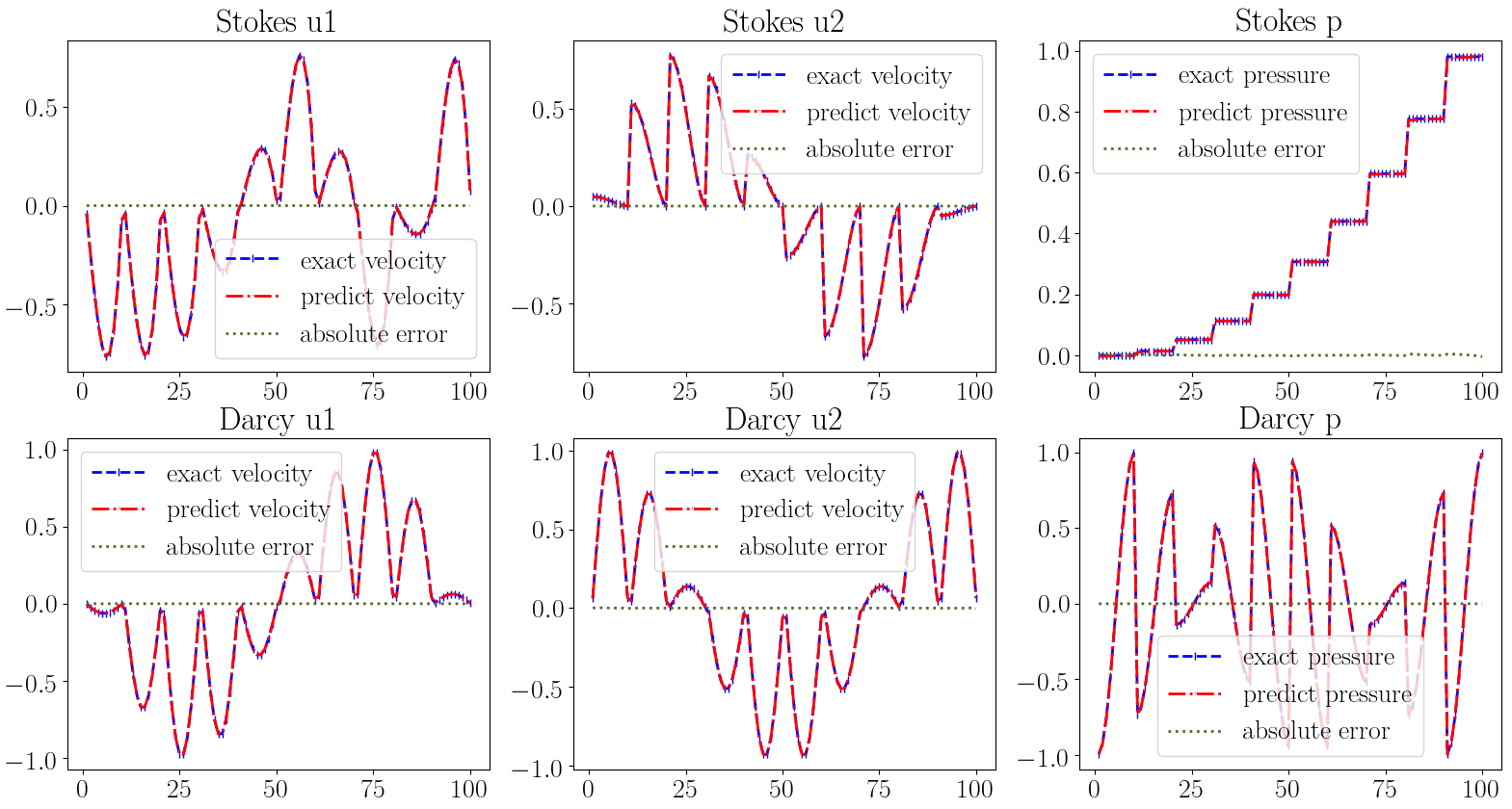}
\caption{ The point-wise errors (Test 3).}
\label{figure7}
\end{figure}

\subsection{Test 4}
The problems that we have studied so far have the exact solution. In this example, we consider coupling of the Stokes and Darcy-Forchheimer problems with no exact solution over $\Omega_S = (-1/2, 3/2)\times (0, 2),~\Omega_D = (-1/2, 3/2)\times (-2, 0)$ and the interface $\Gamma  = \{ -1/2 < x < 3/2,~y = 0\}$.  Specifically, in the Stokes region, the Dirichlet boundary condition is given by Kovasznay flow \cite{L. I. G. Kovasznay},
$$\mathbf{u}_S = \left(
                   \begin{array}{c}
                    1-e^{\lambda x}cos(2\pi y) \\
                    \frac{\lambda}{2\pi}e^{\lambda x}sin(2\pi y)\\
                   \end{array}
                 \right),
$$
where $\lambda  =\frac{- 8\pi^2}{ 1+\sqrt{1+64\pi ^2}}$. Moreover, we set $\mu  = \rho  = \beta  = \nu  = 1$ and $\mathbf{g}_D = \mathbf{0}, ~f_D = 0, ~\mathbf{f}_S = \mathbf{0}$. In addition, $p_D$ satisfies the homogeneous Dirichlet boundary condition along $y = -2$, otherwise it has an homogeneous Neumann boundary condition. The permeability is taken to be $K = \varepsilon I$ and the $\varepsilon=10^{4}$. Since the exact solution for this example is not available, we provide $L^2$ error of interface to demonstrate the accuracy of the CDNNs in Table \ref{table4}. Obviously, the error decreases gradually with the increasing of the hidden layers. Furthermore, Figures \ref{figure8} - \ref{figure9} display the exact solution and the results of CDNNs in detail.
\begin{table}[H]
\centering
\caption{The error in interface of Test 4 (K=10000).}
\begin{tabular}{c@{\extracolsep{0.2em}}c@{\extracolsep{0.2em}}c@{\extracolsep{0.2em}}c@{\extracolsep{0.2em}}c@{\extracolsep{0.2em}}c@{\extracolsep{0.2em}}c}
\hline
  & &400 sampled points        &     \\ \hline
   &$Interface1$         &$Interface2$              & $Interface3$       \\ \hline
$1~layer$    \ \ \  &$6.49\times10^{-2}$   \ \ \ &$9.14\times10^{-2}$   \ \ \     &$3.03\times10^{-2}$ \\\hline
$2~layers$    \ \ \  &$2.31\times10^{-2}$   \ \ \ &$4.74\times10^{-2}$   \ \ \   &$2.38\times10^{-3}$  \\\hline
$3~layers$\ \ \  &$2.12\times10^{-2}$   \ \ \ &$1.28\times10^{-2}$   \ \ \        &$4.94\times10^{-4}$  \\\hline\hline\
\end{tabular}
\label{table4}
\end{table}

\begin{figure}[H]
\centering
\includegraphics[width=6in]{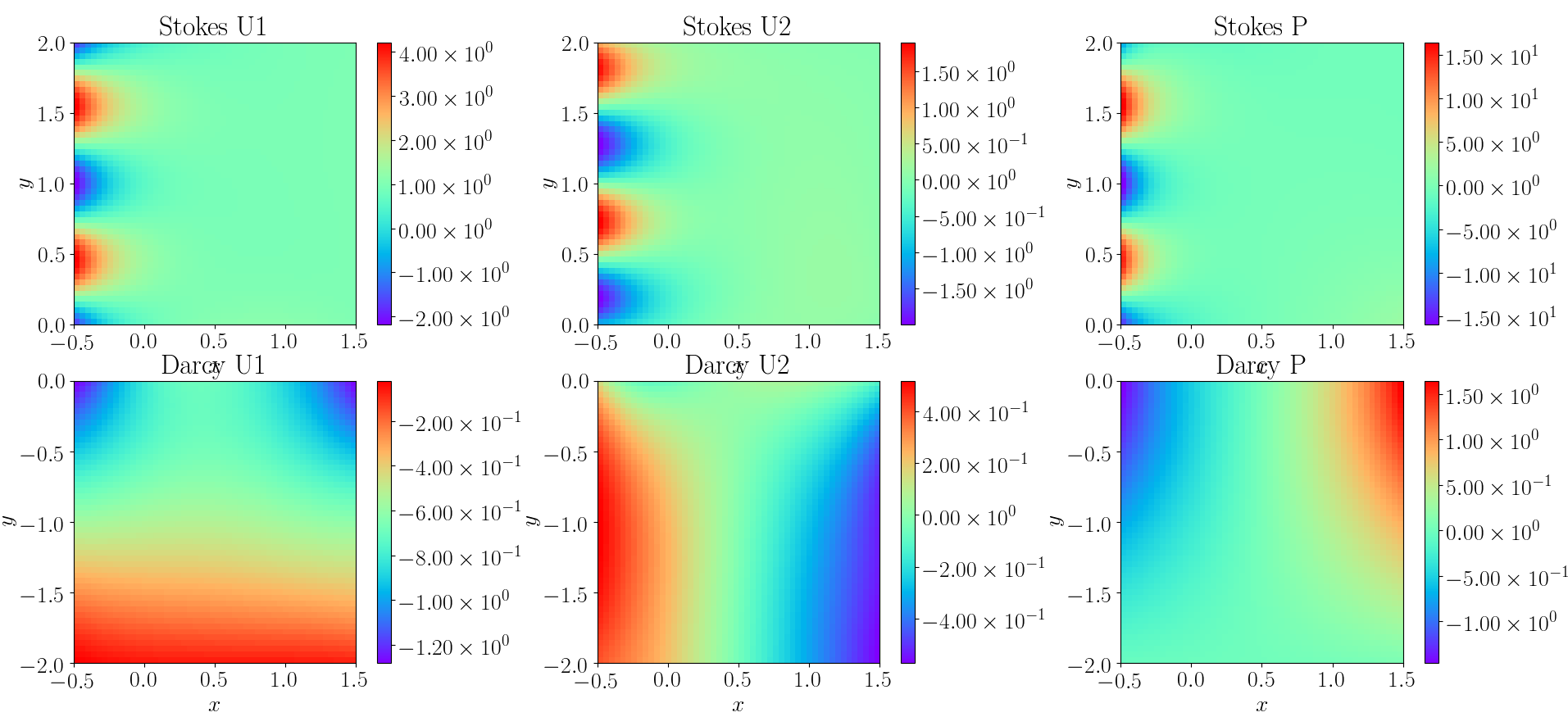}\\
\caption{The results of CDNNs (Test 4).}
\label{figure8}
\end{figure}

\begin{figure}[H]
\centering
\subfigure[Stokes velocity.]{
\begin{minipage}[t]{0.45\linewidth}
\centering
\includegraphics[width=2.1in]{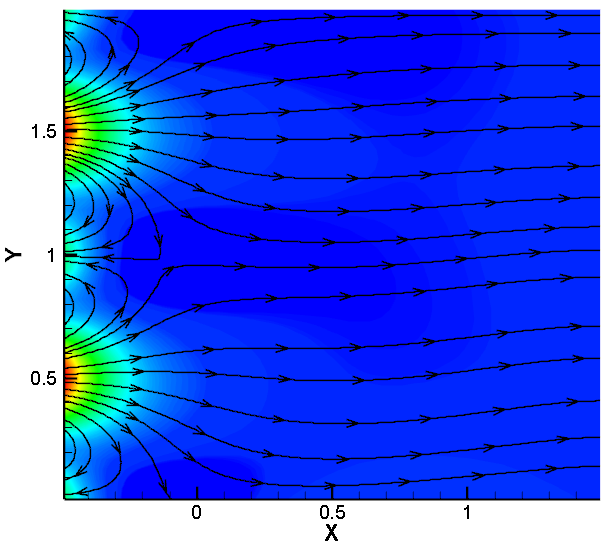}
\end{minipage}}
\subfigure[Darcy velocity.]{
\begin{minipage}[t]{0.45\linewidth}
\centering
\includegraphics[width=2.1in]{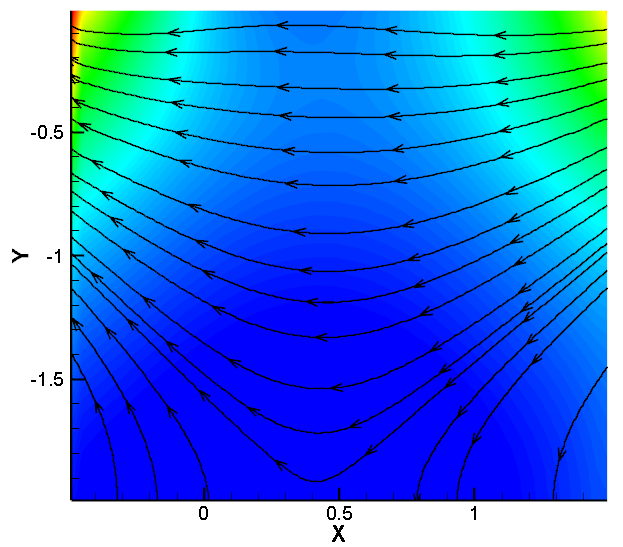}
\end{minipage}}
\caption{The velocity of Stokes and Darcy (Test 4).}
\label{figure9}
\end{figure}

\subsection{Test 5}
We conclude this section with a physical flow, where $\Omega _S = (0, 1) \times (1, 2),~\Omega_ D = (0, 1)^2$ and the interface $\Gamma  = \{ 0 < x < 1,~ y = 1\}$. In $\Omega _S$, the boundaries of the cavity are walls with no-slip condition, except for the upper boundary where a uniform tangential velocity $\mathbf{u}_S(x, 2) = (1, 0)^T$ is imposed, which is driven cavity flow.  And more precisely, we enforce homogeneous Neumann and Dirichlet boundary conditions, respectively, on $\Gamma_ {D,N} = \{ x = 0~ or ~y = 0\}$  and $\Gamma _{D,D} = \{ x = 1\}$. In addition, we set $K$ to be the identity tensor in $\mathbb{R} ^{2\times 2}, ~\mu  = \rho  = \beta  = \nu  = 1$, and $f_D=0,~\mathbf{f}_S =\mathbf{ 0},~\mathbf{g}_ D = \mathbf{0}$. The results of the CDNNs are depicted in Figure \ref{figure10}. More vividly, we display the velocity flows of  free-flow and porous media zones in Figure \ref{figure10}.
\begin{figure}[H]
\centering
\includegraphics[width=5.5in]{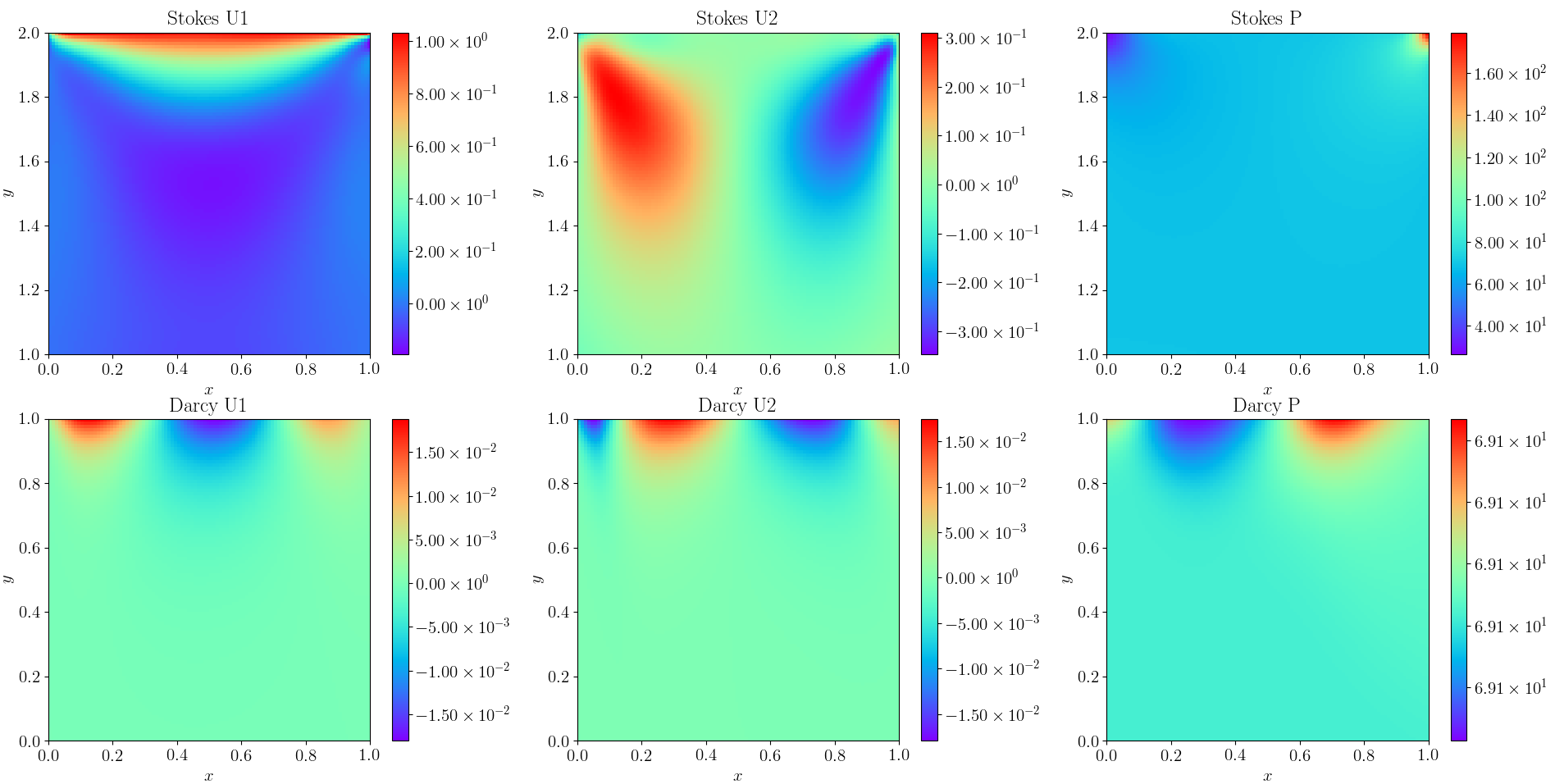}
\caption{The results of the CDNNs (Test 5).}
\label{figure10}
\end{figure}
\begin{figure}[H]
\centering
\subfigure[Stokes velocity.]{
\begin{minipage}[t]{0.45\linewidth}
\centering
\includegraphics[width=1.9in]{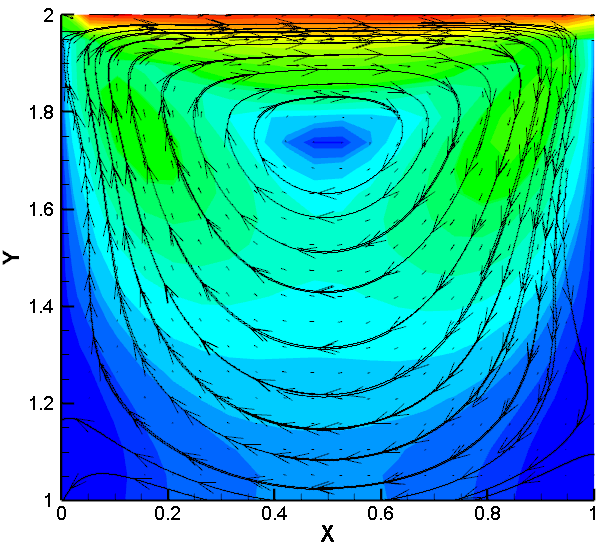}
\end{minipage}}
\subfigure[Darcy velocity.]{
\begin{minipage}[t]{0.45\linewidth}
\centering
\includegraphics[width=1.9in]{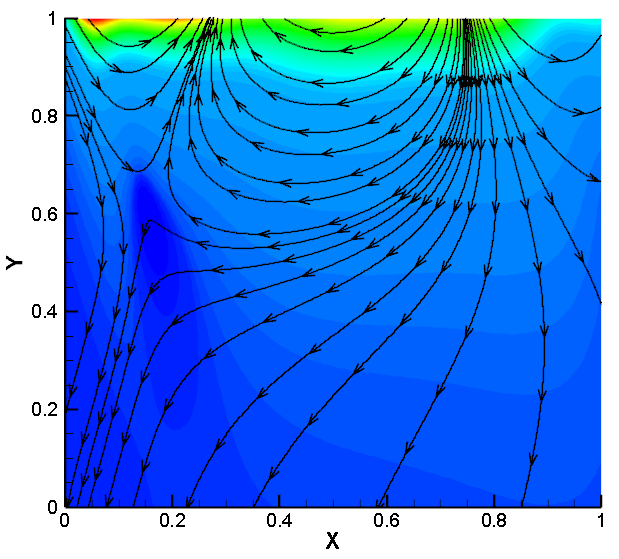}
\end{minipage}}
\caption{The velocity flows of Stokes and Darcy (Test 5).}
\label{figure11}
\end{figure}

\section{Conclusions}
In this article, we proposed the CDNNs to study the coupled Stokes and Darcy-Forchheimer problems. This method can avoid many limitations of the traditional methods, such as decoupling, grid construction and the complicated interface conditions. Specially, we provide the convergence of the loss function and the convergence of the CDNNs to the exact solution. The numerical results are consistent with our theory sufficiently. Moreover, we leave the following issues subject to our future works, 1) Combining data-driven with model-driven to solve the high dimensional coupled problems; 2) Considering the specific size of the networks through theoretical analysis; 3) Combining traditional numerical methods with deep learning to solve more complicated high dimensional coupled problems.

\end{document}